\documentclass[11pt,a4paper]{article}
\usepackage{latexsym, amssymb, amsmath, amscd, amsthm, verbatim, a4, mathdots, mathbbol}
\usepackage{amsfonts,cmmib57,graphics}
\usepackage{graphics}
\usepackage[pdftex]{graphicx}
\usepackage{dsfont}
\usepackage{url}
\usepackage[latin1]{inputenc}
\usepackage[active]{srcltx}
\usepackage{algorithm}
\usepackage{algorithmic}
\usepackage{subfigure}
\usepackage{multicol}
\usepackage[pdftex]{color}
\usepackage{epstopdf}

\makeatletter
\def\revddots{\mathinner{\mkern1mu\raise\p@
    \vbox{\kern7\p@\hbox{.}}\mkern2mu
    \raise4\p@\hbox{.}\mkern2mu\raise7\p@\hbox{.}\mkern1mu}}
\makeatother
\mathcode`@="8000 
{\catcode`\@=\active\gdef@{\mkern1mu}}
\newcommand{\beq}{\begin{equation}}
\newcommand{\eeq}{\end{equation}}
\newcommand{\be}{\begin{enumerate}}
\newcommand{\ee}{\end{enumerate}}
\newcommand{\bi}{\begin{itemize}}
\newcommand{\ei}{\end{itemize}}

\def\C{\mathbb{C}}

\def\AA{\mathcal{A}}
\def\BB{\mathcal{B}}

\def\norm2#1{\|#1\|_2}

\numberwithin{equation}{section}
\numberwithin{table}{section}
\numberwithin{figure}{section}
\newtheorem{theorem}{Theorem}[section]

\newtheorem{lemma}[theorem]{Lemma}

\newtheorem{definition}[theorem]{Definition}
\theoremstyle{definition}

\newtheorem{numexamp}[theorem]{Numerical experiment}
\newtheorem{remark}[theorem]{Remark}
\newcommand{\bmat}{\begin{bmatrix}}
\newcommand{\ebmat}{\end{bmatrix}}
\newcommand{\pmat}{\begin{pmatrix}}
\newcommand{\epmat}{\end{pmatrix}}
\newcommand{\smat}{\begin{smallmatrix}}
\newcommand{\esmat}{\end{smallmatrix}}

\mathchardef\Gamma="7100
\mathchardef\Delta="7101
\mathchardef\Theta="7102
\mathchardef\Lambda="7103
\mathchardef\Xi="7104
\mathchardef\Pi="7105
\mathchardef\Sigma="7106
\mathchardef\Upsilon="7107
\mathchardef\Phi="7108
\mathchardef\Psi="7109
\mathchardef\Omega="710A

\newcommand{\cA}{{\cal A}}
\newcommand{\cB}{{\cal B}}

\newcommand{\cK}{{\cal K}}
\newcommand{\cL}{{\cal L}}

\newcommand{\cO}{{\cal O}}
\newcommand{\cP}{{\cal P}}

\newcommand{\cS}{{\cal S}}

\newcommand{\cU}{{\cal U}}
\newcommand{\cV}{{\cal V}}

\newcommand{\bfR}{{\mathbf R}}

\newcommand{\BC}{{\mathbb C}}

\newcommand{\BR}{{\mathbb R}}

\newcommand{\UU}{{\mathbf U}}
\newcommand{\bfA}{{\mathbf A}}
\newcommand{\bfB}{{\mathbf B}}
\newcommand{\bfAA}{{\mathbf {\mathcal{A}}}}
\newcommand{\bfBB}{{\mathbf {\mathcal{B}}}}
\newcommand{\eq} [1] {\begin{equation}\label{#1}}
\newcommand{\en} {\end{equation}}
\newcommand {\eqn}  {\begin{eqnarray}}
\newcommand {\enn}  {\end{eqnarray}}
\newcommand {\bstar}    {\begin{eqnarray*}}
\newcommand {\estar}    {\end{eqnarray*}}
\newcommand {\mat}  [1] {\left[\begin{array}{#1}}
\newcommand {\rix}      {\end{array}\right]}

\newcommand{\la}{\lambda}

\newif\ifMDlatex
\catcode`@=11     
\def\MD@us#1{\csname#1style\endcsname}
\def\MD@uf#1{\csname#1font\endcsname}
\def\MD@t{text}
\def\MD@s{script}
\def\MD@ss{scriptscript}
\newdimen\MD@unit
\MD@unit\p@
\def\MD@changestyle#1{
  \relax\MD@unit0.1\fontdimen6\MD@uf{#1}0
  \everymath\expandafter{\the\everymath\MD@us{#1}}
}
\def\MD@dot{$\m@th\ldotp$}
\def\MD@palette#1{\mathchoice{#1\MD@t}{#1\MD@t}{#1\MD@s}{#1\MD@ss}}
\def\MD@ddots#1{{\MD@changestyle{#1}%
  \mkern1mu\raise7\MD@unit\vbox{\kern7\MD@unit\hbox{\MD@dot}}%
  \mkern2mu\raise4\MD@unit\hbox{\MD@dot}%
  \mkern2mu\raise \MD@unit\hbox{\MD@dot}\mkern1mu}}%
\def\MD@iddots#1{{\MD@changestyle{#1}%
  \mkern1mu\raise \MD@unit\hbox{\MD@dot}%
  \mkern2mu\raise4\MD@unit\hbox{\MD@dot}%
  \mkern2mu\raise7\MD@unit\vbox{\kern7\MD@unit\hbox{\MD@dot}}}}%
\def\MD@vdots#1{\vbox{\MD@changestyle{#1}%
    \baselineskip4\MD@unit\lineskiplimit\z@
    \kern6\MD@unit\hbox{\MD@dot}\hbox{\MD@dot}\hbox{\MD@dot}}}%
\ifMDlatex
  \DeclareRobustCommand\ddots{\mathinner{\MD@palette\MD@ddots}}%
  \DeclareRobustCommand\iddots{\mathinner{\MD@palette\MD@iddots}}%
  \DeclareRobustCommand\vdots{\mathinner{\MD@palette\MD@vdots}}%
\else
  \def\ddots{\mathinner{\MD@palette\MD@ddots}}%
  \def\iddots{\mathinner{\MD@palette\MD@iddots}}%
  \def\vdots{\mathinner{\MD@palette\MD@vdots}}%
\fi
\catcode`@=12    

\newcommand{\comm}[1]{}
\newcommand {\mycomment}[1]{} 

\mathcode`@="8000 
{\catcode`\@=\active\gdef@{\mkern1mu}}

\begin{document}
\title{A compact rational Krylov method for large-scale rational eigenvalue problems}
%
\author{
        Froil\'{a}n M. Dopico%
        \thanks{%
                Departamento de Matem\'{a}ticas,
                Universidad Carlos III de Madrid,
                Avda.\ Universidad 30, 28911
                Legan\'{e}s, Spain
                ({\tt dopico@math.uc3m.es}).
                Supported by  ``Ministerio de Econom\'{i}a, Industria y Competitividad of Spain'' and ``Fondo Europeo de Desarrollo Regional (FEDER) of EU''
                through grants MTM2012-32542, MTM2015-65798-P and MTM2015-68805-REDT.
               }
   \and
        Javier Gonz\'{a}lez-Pizarro%
        \thanks{%
                Departamento de Matem\'{a}ticas,
                Universidad Carlos III de Madrid,
                Avda.\ Universidad 30, 28911
                Legan\'{e}s, Spain
                ({\tt jagpizar@math.uc3m.es}).
                Supported by ``Ministerio de Econom\'{i}a, Industria y Competitividad of Spain'' and ``Fondo Europeo de Desarrollo Regional (FEDER) of EU''
                through grants MTM2012-32542 and MTM2015-65798-P and by ``Comisión Nacional de Investigación Científica y Tecnológica (CONICYT) de Chile"  through grant BCH 72160331.
               }
               }

\date{\today}
\maketitle

\begin{abstract}
In this work, we propose a new method, termed as R-CORK, for the numerical solution of large-scale rational eigenvalue problems, which is based on a linearization and on a compact decomposition of the rational Krylov subspaces corresponding to this linearization. R-CORK is an extension of the compact rational Krylov method (CORK) introduced very recently in \cite{cork} to solve a family of non-linear eigenvalue problems that can be expressed and linearized in certain particular ways and which include arbitrary polynomial eigenvalue problems, but not arbitrary rational eigenvalue problems. The R-CORK method exploits the structure of the linearized problem by representing the Krylov vectors in a compact form in order to reduce the cost of storage, resulting in a method with two levels of orthogonalization. The first level of orthogonalization works with vectors of the same size that the original problem, and the second level works with vectors of size much smaller than the original problem. Since vectors of the size of the linearization are never stored or orthogonalized, R-CORK is more efficient from the point of views of memory and orthogonalization than the classical rational Krylov method applied directly to the linearization. Taking into account that the R-CORK method is based on a classical rational Krylov method, to implement implicit restarting is also possible and we show how to do it in a memory efficient way. Finally, some numerical examples are included in order to show that the R-CORK method performs satisfactorily in practice.
\end{abstract}

{\small
 \\

}


\section{Introduction} \label{sect.intro}
In this work, we consider the rational eigenvalue problem (REP)

\begin{equation}\label{rateig}
R(\lambda)x=0,
\end{equation}
where $R(\lambda) \in \BC(\lambda)^{n\times n}$ is a nonsingular rational matrix, i.e., the entries of $R(\la)$ are scalar rational functions in the variable $\la$ with complex coefficients and det$(R(\lambda))\not\equiv 0$ is not identically zero, and $x\in \BC^n$ is a nonzero vector. More precisely, we consider that $R(\la)$ is given as
\begin{equation}\label{formrat}
R(\lambda)=P(\lambda)-\sum_{i=1}^k\dfrac{f_i(\lambda)}{g_i(\lambda)}E_i,
\end{equation}
where $P(\lambda) \in \BC[\lambda]^{n\times n}$ is a matrix polynomial of degree $d$ in the variable $\lambda$, $f_i(\lambda)$, $g_i(\lambda)$ are coprime scalar polynomials of degrees $m_i$ and $n_i$, respectively, $m_i < n_i$ and $E_i\in \BC^{n \times n}$ are constant matrices for $i=1,\dots,k$. We emphasize that it is well known that every rational matrix can be written in the form \eqref{formrat} \cite{kailath,rosenbrock-book} (see also \cite[Section 2]{ADMZ2016}) and that such form appears naturally in many applications \cite{SuBai}.

The REP has attracted considerable interest in recent years since it arises in different applications in some fields such as vibration of fluid-solid structures \cite{SolVib}, optimization of acoustic emissions of high speed trains \cite{HighTrain}, free vibration of plates with elastically attached masses \cite{solovev}, free vibrations of a structure with a viscoelastic constitutive relation describing the behavior of a material \cite{visco,mvoss2016}, and electronic structure calculations of quantum dots \cite{quantum2, quantum1}.

A first idea to solve REPs is a brute-force approach, since one can multiply by $\prod_{i=1}^k{g_i(\lambda)}$ to turn the rational matrix \eqref{formrat}  into a matrix polynomial of degree $d+n_1+\cdots+n_k$.  The common approach to solve a polynomial eigenvalue problem (PEP) is via linearization (see, for instance, \cite{PEP2, mackey, visco}), this is, by transforming the PEP into a generalized eigenvalue problem (GEP) and then applying a well-established algorithm to this GEP, as for instance the QZ algorithm in the case of dense medium sized problems \cite{golub-book} or some Krylov subspace method for large-scale problems. However, this brute-force approach it is only useful when $n_1+n_2+\cdots+n_k$ is small compared with $d$.
So, if $k$ or some $n_i$ are big, then the degree of the matrix polynomial associated to the problem is also big, and this makes the size of the linearization too large, which is impractical for medium to large-scale problems. This drawback has motivated the idea of linearizing directly the REP \cite{SuBai}. The linearization for $R(\lambda)$ in \eqref{formrat} constructed in \cite{SuBai} has a size much smaller than the size of the linearization obtained by the brute-force approach. Nonetheless the increase of the size of the problem is still considerable, so for large-scale rational eigenvalue problems, a direct application of this approach, i.e., without taking into account the structure of the linearization, is also impractical. This idea of taking advantage of the structure of the linearization for solving large-scale REPs is closely connected to the intense research effort developed in the last years by different authors for solving large-scale PEPs via linearizations and that is briefly discussed in the next paragraph.

Several methods have been developed to solve large-scale PEPs numerically by applying Krylov methods to the associated GEPs obtained through linearizations. In this approach, the key issues to be solved for using Krylov methods for large-scale PEPs are the increase of the memory cost and the increase of the orthogonalization cost at each step, as a consequence of the increase of the size of the linearization with respect to the size of the original problem. In order to reduce these costs, different representations of the Krylov vectors of the linearizations have been developed.
First, the second order Arnoldi method (SOAR) \cite{SOAR} and the quadratic Arnoldi method (Q-Arnoldi) \cite{QArnoldi} were developed to solve quadratic eigenvalue problems (QEP), introducing a new representation of the Krylov vectors. However, both methods are potentially unstable as a consequence of performing implicitly the orthogonalization. To cure this instability, the two-level orthogonal Arnoldi process (TOAR) \cite{charlaTOAR, toar} for QEP proposed a different compact representation for the Krylov vectors of the linearization and, combining this representation with the linearization and the Arnoldi recurrence relation, resulted in a memory saving and numerically stable method. Extending the ideas of a compact representation of the Krylov vectors and of the two levels of orthogonalization from polynomials of degree 2 (TOAR) to polynomials of any degree, the authors of \cite{Kressner} developed a memory-efficient and stable Arnoldi process for linearizations of matrix polynomials expressed in the Chebyshev basis. In 2015, the compact rational Krylov method (CORK) for nonlinear eigenvalue problems (NLEP) was introduced in \cite{cork}. CORK considers particular NLEPs that can be expressed and linearized in certain ways, which are solved by applying a compact rational Krylov method to such linearizations. A key feature of the CORK method is that it works for many kinds of linearizations involving a Kronecker structure, as the Frobenius companion form or linearizations of matrix polynomials in different bases (as Newton or Chebyshev, among others \cite{Newton}). CORK reduces both the costs of memory and orthogonalization by using a generalization of the compact Arnoldi representation of the Krylov vectors of the linearizations used in TOAR \cite{charlaTOAR,toar}, and gets stability through two levels of orthogonalization as in TOAR.

In this paper, we develop a rational Krylov method that works on the linearization of REPs introduced in \cite{SuBai} to solve large-scale and sparse $n\times n$ REPs. To this aim, we introduce a compact rational Krylov method for REPs (R-CORK). In the spirit of TOAR and CORK, we will work with two levels of orthogonalization, and, as in CORK, we adapt the classical rational Krylov method \cite{ruhe1, ruhe2, cork} on the linearization to a compact representation of the Krylov vectors and to the two levels of orthogonalization. We can perform the shift-and-invert step by solving linear systems of size $n$. To this purpose, the linearization introduced in \cite{SuBai} is preprocessed in a convenient way and, then, an ULP decomposition is used. This decomposition is similar to the one employed in \cite{cork} directly on the linearizations of the NLEPs considered there. Once this step is performed, we start with the two levels of orthogonalization.
The first level involves an orthogonalization process with vectors of size $n$ and in the second level of orthogonalization we work with vectors of size much smaller than $n$, so this level is cheap compared with the first level. As a result, we develop a stable method that allows us to reduce the orthogonalization cost and the memory cost by exploiting the structure of the matrix pencil that linearizes the REP and using the rational Krylov recurrence relation.

The rest of the paper is organized as follows. Section \ref{sec.prelim} introduces some preliminary concepts: a summarized background on polynomial and rational eigenvalue problems, the classical rational Krylov method for the generalized eigenvalue problem, and the CORK method particularized to polynomial eigenvalue problems.
Section \ref{sect.newTOAR} proposes the compact rational Krylov decomposition that we use to develop the R-CORK method and presents the detailed algorithm with the two levels of orthogonalization. Section \ref{sect.implicit} discusses the implementation of implicit restarting for the R-CORK method.
Section \ref{sec.numexper} presents numerical examples which show that the R-CORK method works satisfactorily in practice and, finally, the main conclusions and some lines of future research are discussed in Section \ref{sec.conclu}.

{\bf Notation.} We denote vectors by lowercase characters, $u$, and matrices by capital characters, $A$. Block vectors and block matrices are denoted by bold face fonts, $\mathbf{u}$, and $\mathbf{A}$, respectively, and the $i$-th block of $\mathbf{u}$ is represented by $u^{(i)}$. The conjugate transpose of $A$ is denoted as $A^*$.  The $i \times j$ matrix with the main diagonal entries equal to 1 and the rest of entries equal to zero is represented by $I_{i\times j}$. In the particular case of $i=j$ this matrix is the identity matrix and is denoted by $I_i$. The vector $e_j$ represents the canonical vector associated to the $j$-th column of the identity matrix and $0_{i\times j}$ represents the zero matrix of size $i\times j$, which in the particular case $i=j$ is denoted simply by $0_j$. The matrix $U_j$ represents a matrix with $j$ columns and $u_i$ represents the $i$-th column of $U_j$.
The rational Krylov subspace of order $m$ associated with the matrices $A$ and $B \in \BC^{n\times n}$, the initial vector $u_1\in \BC^n$ and the shifts $\theta_1,\theta_2,\dots,\theta_{m-1} \in \BC$ is denoted by
 \begin{equation}\label{rat.subs}
 \mathcal{K}_m(A,B,u_1,\theta_{1,\dots,m-1})=\mbox{span}\{ u_1, (A-\theta_1 B)^{-1}Bu_1,(A-\theta_2 B)^{-1}Bu_2,\dots,(A-\theta_{m-1}B)^{-1}Bu_{m-1}\}
 \end{equation}
where $u_{i+1}=(A-\theta_{i}B)^{-1}Bu_{i}$, $i=1,\dots,m-1$. We omit subscripts when the
dimensions of the matrices are clear from the context. The norm  $\| \cdot \|_2 $ represents the 2-norm and $\| \cdot \|_F$ the Frobenius norm \cite[Ch. 5]{Horn}. The Kronecker product of two matrices is denoted by $A\otimes B$. The set of $n\times n$ rational matrices is denoted by $\mathbb{C}(\la)^{n\times n}$ and
the set of $n\times n$ polynomial matrices (or, equivalently, matrix polynomials) is denoted by $\mathbb{C}[\la]^{n\times n}$.

\section{Preliminaries} \label{sec.prelim}
\subsection{Basics on polynomial eigenvalue problems and linearizations}
The classical approach to solve a regular PEP
\begin{equation}\label{pep}
P(\lambda)x=0,
\end{equation}
where $P(\lambda)=\sum_{i=0}^d{\lambda^iP_i}$ with $P_i\in\BC^{n\times n}$ and det$(P(\lambda))\not\equiv 0$ is via linearization. In this process, the matrix polynomials are mapped into matrix pencils with the same eigenvalues and multiplicities \cite{Lancaster, mackey}. More precisely, a pencil $L(\lambda)={\mathbf{A}}-\lambda \mathbf{B}$ is called a linearization of $P(\lambda)$ if there exist unimodular matrix polynomials\footnote{Unimodular matrix polynomials are matrix polynomials whose determinant is a nonzero constant, i.e., it does not depend of $\lambda$. Most of the linearizations considered in this work are in fact strong linearizations \cite{PEP2,mackey}, so, they preserve also the eigenvalues at infinity of $P(\lambda)$, and their multiplicities, if they are present. Nevertheless in this work we do not intend to compute infinite eigenvalues since their existence is not generic, and so we do not need to use the concept of strong linearization. } $E_1(\lambda)$, $E_2(\lambda)$ such that
\begin{equation*}
\left[ \begin{array}{cc} P(\lambda) & 0 \\ 0 & I_{(d-1)n} \end{array}\right]=E_1(\lambda)({\mathbf{A}}-\lambda \mathbf{B})E_2(\lambda).
\end{equation*}

Some linearizations of matrix polynomials of degree $d$ and size $n\times n$, very useful in practice, are of the form as the pencils in Definition \ref{strucpencil}.
\begin{definition}\cite[Definition 2.2]{cork}\label{strucpencil}
Let $P(\lambda) \in \BC[\lambda]^{n\times n}$ be a regular matrix polynomial, i.e.,  {\rm det}$(P(\lambda))$ does not vanish identically, of degree $d\ge 2$ and size $n\times n$.
A $dn \times dn$ matrix pencil $L(\lambda)$ of the form
\begin{equation}\label{LEP}
L(\lambda)=\mathbf{A}-\lambda \mathbf{B},
\end{equation}
where
\begin{equation}\label{pencilpep}
\bfA=\left[\dfrac{A_0 \quad A_1 \quad \cdots \quad A_{d-1}}{M\otimes I_n}\right], \quad
\bfB=\left[\dfrac{B_0 \quad B_1 \quad \cdots \quad B_{d-1}}{N\otimes I_n}\right]
\end{equation}
and $A_i$, $B_i\in \BC^{n\times n}$, $i=0,1,\dots,d-1$, and $M$, $N\in\BC^{(d-1)\times d}$, is called a structured linearization pencil of $P(\lambda)$ if the following conditions hold
\begin{enumerate}
\item $L(\lambda)$ is a linearization of $P(\lambda)$,
\item $M-\lambda N$ has rank $d-1$ for all $\lambda\in\BC$, and
\item $(\bfA- \lambda \bfB)(f(\lambda) \otimes I_n)=e_1\otimes P(\lambda)$ for some polynomial function $f:\BC \rightarrow \BC[\lambda]^{d}$, $f(\lambda)\neq 0$ for all $\lambda\in\BC$, where $e_1\in \BC^d$ is the first vector of the canonical basis of $\BC^d$.
\end{enumerate}
\end{definition}
The matrices $A_i$ and $B_i$ that appear in the first block rows in \eqref{pencilpep}  are related to the matrix polynomial $P(\lambda)$ and the matrices $M$ and $N$ correspond to the linear relations between the basis functions $f_i(\lambda)$, where $f(\lambda) := [f_1(\lambda), \ldots, f_d(\lambda)]^T$, used in the representation of the matrix polynomial. The interested reader can find some examples in \cite{cork}. The identity $(\bfA-\lambda \bfB)(f(\lambda)\otimes I_n)=e_1\otimes P(\lambda)$ generalizes the identity used in \cite{mackey} to define certain vector spaces of linearizations of matrix polynomials.

An important property of structured linearization pencils is that their eigenvectors are closely related to the eigenvectors of the matrix polynomial as we can see in Theorem \ref{th.eigenv}.
\begin{theorem}{\cite[Corollary 2.4]{cork}}\label{th.eigenv}
Let $L(\lambda)$ be a structured linearization pencil of $P(\lambda)$ as in Definition \ref{strucpencil} and let $(\lambda_\star,\mathbf{x})$ be an eigenpair of $L(\lambda)$. Then, the eigenvector $\mathbf{x}$ has the following structure
\begin{equation*}
\mathbf{x}=f(\lambda_\star)\otimes x,
\end{equation*}
where $x\in \BC^n$ is an eigenvector of $P(\lambda)$ corresponding to $\lambda_\star$.
\end{theorem}
The block ULP decomposition in Theorem \ref{ULP} for structured linearization pencils of matrix polynomials is important for the CORK method introduced in \cite{cork} because it allows to perform the shift-and-invert step in CORK efficiently. We will use also a decomposition of this type to perform the shift-and-invert step in the R-CORK method developed in Section \ref{sect.newTOAR}.
\begin{theorem}{\cite[Theorem 2.3]{cork}}\label{ULP}
Let $\bfA$ and $\bfB$ be defined by \eqref{pencilpep}. Then, for every $\mu\in \BC$ there exists a permutation matrix $\cP \in \BC^{d\times d}$ such that the matrix $(M_1-\mu N_1)\in \BC^{(d-1)\times (d-1)}$ is invertible with
\begin{equation*}
M=:[m_0 \quad M_1]\cP, \quad N=:[n_0\quad N_1]\cP.
\end{equation*}
Moreover, the matrix $L(\mu)$, i.e., the pencil $L(\la)$ in \eqref{LEP} evaluated in $\mu$, can be factorized as follows
\begin{equation*}
L(\mu)=\bfA-\mu \bfB= \mathcal{U}(\mu)\mathcal{L}(\mu)(\cP\otimes I_n),
\end{equation*}
where
\begin{eqnarray*}
\mathcal{L}(\mu) &=& \left[ \begin{array}{cc} P(\mu) & 0 \\ (m_0-\mu n_0)\otimes I_n & (M_1 -\mu N_1)\otimes I_n \end{array} \right], \\
\mathcal{U}(\mu) &=& \left[ \begin{array}{cc} \alpha^{-1}I_n & (\mathbf{\bar{A}}_1-\mu \mathbf{\bar{B}}_1)((M_1-\mu N_1)^{-1}\otimes I_n) \\ 0 & I_{(d-1)n} \end{array} \right],
\end{eqnarray*}
with the scalar $\alpha=e_1^T\cP f({\mu})\neq 0$ and
\begin{eqnarray*}
[A_0 \quad A_1 \quad \cdots \quad A_{d-1}] &=:& [\bar{A}_0 \quad \mathbf{\bar{A}}_1](\cP\otimes I_n), \\
{[}B_{0} \quad B_{1} \quad \cdots \quad  B_{d-1}{]} &=:& {[} \bar{B}_0 \quad \mathbf{\bar{B}}_1 {]}(\cP \otimes I_n).
\end{eqnarray*}
\end{theorem}

\subsection{A linearization for rational eigenvalue problems}
In this subsection, we present some results and notations related to the REP. Interested readers can find more information in the summaries presented in \cite[Sections 1 \& 2]{behera} and \cite[Section 2]{ADMZ2016}, as well as in the classical references \cite{kailath,rosenbrock-book}.

In this work, we assume that the rational matrix $R(\lambda)$ in \eqref{formrat} is regular, this means, $\mbox{det}(R(\lambda)) \not\equiv 0$. With a slight lack of rigor, we can say that if the matrices $E_i$ in \eqref{formrat} are linearly independent, then the roots of the denominators $g_i(\lambda)$ are the poles of $R(\lambda)$ and that $R(\lambda)$ is not defined in these poles. A scalar $\lambda \in \mathbb{C}$ which is not a pole is called an eigenvalue of $R(\lambda)$ if det$(R(\lambda))=0$, and a nonzero vector $x \in \mathbb{C}^n$ is called an eigenvector of $R(\lambda)$ associated to the eigenvalue $\lambda$ if the condition \eqref{rateig} holds. The pair $(\lambda,x)$ constitutes an eigenpair of $R(\lambda)$ and our goal is to compute a subset of such eigenpairs.

We express the matrix polynomial $P(\lambda)$ of degree $d$ in \eqref{formrat} as follows
\begin{equation}\label{matrixpol}
P(\lambda)=\lambda^dP_d+\lambda^{d-1}P_{d-1}+\cdots+\lambda P_1+P_0,
\end{equation}
where $P_i \in \BC^{n \times n}$ for $i=0,\dots,d$. From now on, we assume the generic condition that the leading coefficient matrix $P_d$ is nonsingular in \eqref{matrixpol}. As explained in the introduction, we assume that $f_i(\lambda)$ and $g_i(\lambda)$ in \eqref{formrat} are coprime, this is, they do not have common factors, and that the rational functions $\frac{f_i(\lambda)}{g_i(\lambda)}$ are strictly proper, this is, the degree, $m_i$, of $f_i(\lambda)$ is smaller than the degree, $n_i$, of $g_i(\lambda)$.
Under these assumptions, in \cite{SuBai}, Su and Bai proposed a linearization to solve the rational eigenvalue problem. With this aim, they first showed that one can find matrices $E,F$ of size $n\times s$, and matrices $C,D$ of size $s\times s$, with $s=r_1n_1+r_2n_2+\cdots +r_kn_k$ with $r_i=$rank$(E_i)$ in \eqref{formrat}, such that
\begin{equation}\label{newreprat}
R(\lambda)=P(\lambda)-E(C-\lambda D)^{-1}F^T.
\end{equation}
In fact, it is a classical result (much older than \cite{SuBai}) that any rational matrix can be written as in \eqref{newreprat} by expressing $R(\la)$ as the sum of its unique polynomial and strictly proper parts, and, then, constructing a state-space realization of the strictly proper part \cite{rosenbrock-book} (see also \cite[Section 2]{ADMZ2016}). However, we emphasize that, as far as we know, \cite{SuBai} is the first reference available in the literature that uses \eqref{newreprat} with the purpose of computing the eigenvalues of a REP, as well as that \cite{SuBai} is the first reference that points out that the representation \eqref{newreprat} is immediately available from the data in many practical REPs without any computational cost.

Once the representation \eqref{newreprat} for the REP is available, the authors of \cite{SuBai} linearized the REP  $R(\lambda)x=0$ as follows:
\begin{equation}\label{linearrat}
(\bfAA-\lambda\bfBB)\mathbf{z}=0,
\end{equation}
where
\begin{equation}\label{matAB}
\mathcal{A}=\left[ \begin{array}{cccc|c}
P_{d-1} & P_{d-2} & \cdots & P_0 & E \\
-I_n & 0 & \cdots & 0 & \\
& \ddots & \ddots & \vdots & \\
& & -I_n & 0 & \\ \hline
& & & F^T & C \end{array}\right], \quad
\mathcal{B}=-\left[ \begin{array}{cccc|c}
P_d & & & & \\
& I_n & & & \\
& & \ddots & & \\
& & & I_n & \\ \hline
& & & & -D \end{array} \right]
\end{equation}
and
\begin{equation}\label{eigenv}
\mathbf{z}=\left[ \begin{array}{c}
\lambda^{d-1}x \\
\lambda^{d-2}x \\
\vdots \\
x \\ \hline
y \end{array}\right], \qquad \mbox{with $y= - (C-\la D)^{-1} \, F^T x$.}
\end{equation}
Denoting by $\mathbf{A}$ and $\mathbf{B}$ the upper left $nd\times nd$ submatrices of $\AA$ and $\BB$, we can write $\cA- \lambda \cB$  as follows
\begin{equation}\label{ABreduc}
\cA- \lambda \cB=\left[\begin{array}{c|c} \mathbf{A}-\lambda \mathbf{B} & e_1\otimes E \\ \hline e_d^T\otimes F^T & C-\lambda D\end{array}\right],
\end{equation}
where $e_1$ and $e_d$ are the first and the last columns of $I_d$, respectively. Observe that $\mathbf{A} - \lambda \mathbf{B}$ is the famous first Frobenius companion linearization of the matrix polynomial $P(\la)$ in \eqref{matrixpol} \cite{Lancaster}.

As mentioned above, it is important to remark that in many applications of REPs \cite{visco,SuBai,quantum1}, the first step in the process above, i.e., to construct the representation \eqref{newreprat}, does not involve any computational effort, since the matrices $P_0 , P_1, \ldots, P_d, E, C, D,$ and $F$ can be obtained directly from the data. As a consequence the linearization $\cA - \la \cB$ above can be constructed also without any computational effort and all the computational effort is attached to the solution of the GEP \eqref{linearrat}.  Another important remark that has a deep computational impact is that in many applications of REPs \cite{visco,SuBai,quantum1} the size $s\times s$ of the matrices $C$ and $D$ is much smaller than the size $n\times n$ of the original REP, i.e., $s \ll n$ or, in plain words, the rank of the strictly proper part of $R(\la)$ is much smaller than the size of $R(\la)$. Therefore, if $s \ll n$, then the size $(nd+s) \times (nd+s)$ of the linearization $\AA - \la \BB$ is approximately equal to the size  $(nd) \times (nd)$ of the linearization $\mathbf{A} - \lambda \mathbf{B}$ of the matrix polynomial $P(\la)$, and the costs of solving the GEPs $\AA - \la \BB$ and $\mathbf{A} - \lambda \mathbf{B}$ are expected to be similar. Thus, it is not surprising that the R-CORK algorithm developed in Section \ref{sect.newTOAR} for large-scale REPs is particularly efficient in terms of storage and orthogonalization costs when $s \ll n$. However, we emphasize in this context that R-CORK also improves significantly these costs when $s \approx n$ (and $d \geq 2$) with respect to a direct application of large-scale eigensolvers to $\AA - \la \BB$. We will make often comments about the advantages of considering $s\ll n$ throughout the paper.

A formal definition of linearization of a rational matrix can be found in \cite{behera} and another one which includes the concept of strong linearization in \cite{ADMZ2016}. In fact, it is proved in \cite{ADMZ2016} that $\cA - \la \cB$ in \eqref{linearrat} is a strong linearization of $R(\la)$ in \eqref{newreprat} whenever $-E(C-\la D)^{-1} F^T$ is a {\em minimal} state-space realization \cite{rosenbrock-book} of the strictly proper part of $R(\la)$ \cite[Section 8]{ADMZ2016}. We emphasize that the requirement that $-E(C-\la D)^{-1} F^T$ is a minimal state-space realization is very mild \cite[Section 8]{ADMZ2016} and that is fully necessary to guarantee that for every eigenvalue $\la$ of $R(\la)$ the matrix $C-\la D$ is nonsingular \cite[Example 3.2]{ADMZ2016}. In the rest of the paper we implicitly assume that $E(C-\la D)^{-1}F^T$ is a minimal realization, although the only result we will use explicitly is Theorem \ref{th.rateig}, which remains valid even when $E(C-\la D)^{-1} F^T$ is not minimal. The subtle point is that Theorem \ref{th.rateig} assumes that $\la$ is a number such that the matrix $(C-\lambda D)$ is invertible, but if $E(C-\la D)^{-1}F^T$ is not minimal then there may be eigenvalues of $R(\la)$ that do not satisfy such assumption.

\begin{theorem}{\cite[Theorem 3.1]{SuBai}}\label{th.rateig}
Let $\lambda \in \BC$ be such that {\rm det}$(C-\lambda D)\neq 0$. Then the following statements hold:
\begin{enumerate}
\item[(a)] If $\lambda$ is an eigenvalue of the REP \eqref{newreprat}, then it is an eigenvalue of the GEP \eqref{linearrat}.
\item[(b)] Let $\lambda$ be an eigenvalue of the GEP \eqref{linearrat} and $z=[z_1^T, z_2^T,\cdots,z_d^T,y^T]^T$ be a corresponding eigenvector, where $z_i$ are vectors of length $n$ for $i=1,2,\dots, d$, and $y$ is a vector of length $s$. Then $z_d \neq 0$ and $R(\lambda)z_d=0$, namely, $\lambda$ is an eigenvalue of the REP \eqref{newreprat} and $z_d$ is a corresponding eigenvector. Moreover, the algebraic and geometric multiplicities of $\lambda$ for the REP \eqref{newreprat} and GEP \eqref{linearrat} are the same.
\end{enumerate}
\end{theorem}

Part (b) of Theorem \ref{th.rateig} is the key result that allows us to get the eigenvalues and eigenvectors of $R(\la)$ from those of the GEP $\cA - \la \cB$ in \eqref{linearrat}. In fact, observe that Theorem \ref{th.rateig}-(b) can be improved if $\la \ne 0$, since in this case, according to \eqref{eigenv}, every $z_i$ is an eigenvector of $R(\la)$. Therefore, we have $d$ degrees of freedom for the recovery of the eigenvector of $R(\la)$. The most sensible option from the point of view of rounding errors is to choose the $z_i$ with largest 2-norm, that is, $z_1$ when $|\la| > 1$ and $z_d$ when $|\la| \leq 1$.

The final comment of this section is that in contrast to the CORK method developed in \cite{cork} for PEPs, which is valid for many linearizations, the rational CORK method, R-CORK, introduced in this manuscript uses only the linearization $\cA - \la \cB$ in \eqref{linearrat}. The reason of this restriction is that the theory of linearizations of REPs is far less developed than the theory of linearizations of PEPs. Thus, although many linearizations of REPs have been introduced very recently in \cite{behera,ADMZ2016}, their properties are not yet fully understood.


\subsection{The classical rational Krylov method for generalized eigenvalue problems}
We revise in this subsection the rational Krylov method for GEPs since the algorithm R-CORK presented in this paper is based on this method.
The rational Krylov method \cite{ruhe1, ruhe2} is a generalization for computing eigenvalues of matrices and of matrix pencils of the shift-and-invert Arnoldi method. The main differences between these methods are basically two: in rational Krylov methods we can change the shift $\theta_j$ at each iteration instead of fixing the shift as in the shift-and-invert Arnoldi method. Also, the information of the approximate eigenvalues is contained in two upper Hessenberg matrices $\underbar{$H$}_j$ and $\underbar{$K$}_j$ instead of in only one matrix. In Algorithm \ref{rat.krylov} we present a basic pseudocode of the rational Krylov method that summarizes its main steps and guides the developments in the rest of this subsection, which are a very brief sketch of the rational Krylov method. The reader can find more details in \cite{ruhe1, ruhe2, cork}.
\begin{algorithm}[h]
\caption{Rational Krylov method}\label{rat.krylov}
\begin{algorithmic}
\REQUIRE $\AA$ and $\BB$ square matrices and an initial vector $\mathbf{u}_1$ with $\|\mathbf{u}_1\|_2=1$.
\ENSURE The matrix $\UU_{m+1}$ whose columns are an orthonormal basis of $\mathcal{K}_{m+1}(\AA,\BB,\mathbf{u}_1,\theta_{1,\dots,m})$, and the Ritz pairs $(\lambda,\mathbf{x})$ of $\AA -\lambda \BB$, corresponding to the rational Krylov subspace $\cK_{m}(\AA, \BB, \mathbf{u}_1, \theta_{1,\dots,m-1})$.\\
Initialize $\mathbf{U}_1 = [\mathbf{u}_1]$.
\FOR{$j=1,2,\ldots,m$}
\STATE 1. Choose the shift $\theta_j$.
\STATE 2. $\mathbf{\hat{u}}=(\bfAA-\theta_j \bfBB)^{-1}\bfBB \mathbf{u}_j$.
\STATE 3. $h_j=\mathbf{U}_j^*\mathbf{\hat{u}}$.
\STATE 4. $\mathbf{\tilde{u}}=\mathbf{\hat{u}}-\mathbf{U}_jh_j$.
\STATE 5. Compute the new vector $\mathbf{u}_{j+1}=\mathbf{\tilde{u}}/ h_{j+1,j}$ with $h_{j+1,j}=\|\mathbf{\tilde{u}}\|_2$.
\STATE 6. Update $\mathbf{U}_{j+1}=[ \mathbf{U}_j \quad \mathbf{u}_{j+1}]$.
\STATE 7. Compute the eigenpairs $(\lambda_i,t_i)$ of \eqref{solveHK} and test for convergence.
\ENDFOR
\STATE \quad 8.  Compute the eigenvectors $\mathbf{x}_i=\mathbf{U}_{j+1}\underline{H}_jt_i$, $i=1,\dots,j$.
\end{algorithmic}
\end{algorithm}
This method produces an orthonormal basis for the subspace $\mathcal{K}_{m+1}(\AA,\BB,\mathbf{u}_1,\theta_{1,\dots,m})$.  By using the equalities for $\mathbf{\hat{u}}$ and $\mathbf{\tilde{u}}$ from steps $2$ and $4$ in Algorithm \ref{rat.krylov} we obtain at the $i$-th iteration:
\begin{equation*}
(\bfAA -\theta_i \bfBB)^{-1}\bfBB\mathbf{u}_i=\mathbf{U}_{i+1}\underbar{$h$}_i,
\end{equation*}
with $\underbar{$h$}_i=[h_i^* \quad h_{i+1,i}]^*$. After $j$ steps of the rational Krylov method, we obtain the classic rational Krylov recurrence relation \cite{ruhe2}:
\begin{equation}\label{rec.relation}
\AA \UU_{j+1}\underbar{$H$}_j=\BB \UU_{j+1}\underbar{$K$}_j,
\end{equation}
where $\underbar{$H$}_j,\underbar{$K$}_j \in \BC^{(j+1)\times j}$ are upper Hessenberg matrices and
\begin{equation}\label{defK}
\underbar{$K$}_j=\underbar{$H$}_j\mbox{diag}(\theta_1,\theta_2,\dots,\theta_j)+I_{(j+1)\times j}.
\end{equation}
For simplicity, we assume that breakdowns do not occur in the rational Krylov method, this is, $h_{i+1,i} \neq 0$ for all $i=1,\dots, m$, and in this case the upper Hessenberg matrix $\underline{H}_j$ is unreduced. We can approximate in each iteration of Algorithm \ref{rat.krylov} the corresponding $j$ eigenvalues and eigenvectors of the pencil $\AA-\lambda \BB$ by solving the small generalized eigenvalue problem:
\begin{equation}\label{solveHK}
K_jt_i=\lambda_i H_j t_i, \quad t_i \neq 0,
\end{equation}
where $H_j$ and $K_j$ are the $j \times j$ upper Hessenberg matrices obtained by removing the last rows of $\underbar{$H$}_j$ and $\underbar{$K$}_j$, respectively.
Then, we call $(\lambda_i,\mathbf{x}_i=\UU_{j+1}\underbar{$H$}_jt_i)$ a Ritz pair of $(\cA,\cB)$. We emphasize that the approximate eigenvectors $\mathbf{x}_i$ are not computed in each iteration, since this would be very expensive, and that the test for convergence in step 7 of Algorithm \ref{rat.krylov} can be performed in an inexpensive way by using only the small vectors $t_i$, as it is done in most Krylov methods.


\subsection{The CORK method for polynomial eigenvalue problems}
Van Beeumen, Meerbergen, and Michiels in \cite{cork} proposed a method based on a compact rational Krylov decomposition, extending the two levels of orthogonalization idea of TOAR from the quadratic eigenvalue problem \cite{charlaTOAR, toar} to arbitrary degree polynomial eigenvalue problems and to other NLEPs, including many other linearizations apart from the Frobenius one used in \cite{charlaTOAR,toar}, and using the rational Krylov method instead of the Arnoldi method. This method was baptized as CORK in \cite{cork} and for simplicity we described it particularized to PEPs of degree $d$. The key idea in \cite{cork} is to apply the rational Krylov method in Algorithm \ref{rat.krylov} to a structured linearization pencil of a matrix polynomial $P(\lambda)$ of degree $d$ (recall Definition \ref{strucpencil}) taking into account that the special structure of these pencils imposes a special structure on the bases of the corresponding rational Krylov subspaces. By using this structure, the authors of \cite{cork} reduced both the memory cost and the orthogonalization cost of the classical rational Krylov method applied to an arbitrary pencil of the same size. Considering the matrices $\bfA$ and $\bfB$ in \eqref{pencilpep} and the rational Krylov recurrence relation \eqref{rec.relation} for $\bfA$ and $\bfB$, the authors of \cite{cork} partitioned conformably the matrix $\UU_{j+1}$ as follows
\begin{equation*}
\UU_{j+1}=[\UU_j \quad \mathbf{u}_{j+1}]=\left[ \begin{array}{cc} U_{j}^{(1)} & u_{j+1}^{(1)} \\ U_{j}^{(2)} & u_{j+1}^{(2)} \\ \vdots & \vdots \\ U_{j}^{(d)} & u_{j+1}^{(d)} \end{array}\right],
\end{equation*}
and then, they constructed a matrix $Q_j\in \BC^{n\times r_j}$ with orthonormal columns such that
\begin{equation}\label{defQj}
\mbox{span}\{Q_j\}=\mbox{span}\{U_j^{(1)},U_j^{(2)},\dots,U_j^{(d)}\}
\end{equation}
and rank$(Q_j)=r_j$.
By using the matrix $Q_j$, the blocks $U_j^{(i)}$ for $i=1,2,\dots, d$ can be represented as follows
\begin{equation*}
U_j^{(i)}=Q_jR_j^{(i)}, \quad i=1,2,\dots,d,
\end{equation*}
for some matrices $R_j^{(i)}\in \BC^{r_j\times j}$. Then,
\begin{equation}\label{blockU}
\UU_j=\left[ \begin{array}{c} Q_jR_j^{(1)} \\ Q_jR_j^{(2)} \\ \vdots \\ Q_jR_j^{(d)} \end{array} \right]=\left[ \begin{array}{cccc} Q_j & & & \\ & Q_j & & \\ & & \ddots & \\ & & & Q_j\end{array}\right]\left[ \begin{array}{c} R_j^{(1)} \\ R_j^{(2)} \\ \vdots \\ R_j^{(d)} \end{array} \right] = (I_d\otimes Q_j)\mathbf{R}_j,
\end{equation}
where
\begin{equation*}
\mathbf{R}_j:=\left[ \begin{array}{c} R_j^{(1)} \\ R_j^{(2)} \\ \vdots \\ R_j^{(d)} \end{array} \right].
\end{equation*}
By using this representation, the rational Krylov recurrence relation \eqref{rec.relation} can be written as follows \cite[eq. (4.3)]{cork}
\begin{equation*}
\mathbf{A} (I_d \otimes Q_{j+1}) \mathbf{R}_{j+1} \underbar{$H$}_j = \mathbf{B} (I_d \otimes Q_{j+1}) \mathbf{R}_{j+1} \underbar{$K$}_j.
\end{equation*}
Observe that $\mathbf{U}_j$ has $ndj$ entries while the representation in \eqref{blockU} involves $(n+jd)r_j$ parameters. Therefore, taking into account that in the solution of large-scale PEPs the dimension $j$ of the rational Krylov subspaces is much smaller than the dimension $n$ of the problem and that the degree $d$ of applied PEPs is a low number (for sure smaller than $30$, see \cite{Kressner}, and often much smaller than $30$ \cite{NLEP-collectio}), we get that $jd\ll n$ and that the representation \eqref{blockU} of $\mathbf{U}_j$ stores approximately $n r_j$ numbers. The fundamental reason why the representation of $\mathbf{U}_j$ in \eqref{blockU} is of interest and is indeed compact is because $r_j$ is considerably much smaller than $jd$ for the matrices $\mathbf{A}$ and $\mathbf{B}$ in \eqref{pencilpep}. More precisely, the following result is proved in \cite{cork}.
\begin{theorem}{\cite[Theorems 4.4 and 4.5]{cork}}\label{Qjcork}
Let $Q_j$ be defined as in \eqref{defQj}. Then
\begin{equation}\label{thQj1}
\mbox{span}\{Q_{j+1}\}=\mbox{span}\{Q_j,u_{j+1}^{(p)}\},
\end{equation}
where $u_{j+1}^{(p)}$ represents the block of the vector $\mathbf{u}_{j+1}$ in a certain p-th position determined in \cite{cork}. Also,
\begin{equation}\label{theorj}
r_j<j+d.
\end{equation}
\end{theorem}
Note that Theorem \ref{Qjcork} shows that $Q_j$ can be expanded to $Q_{j+1}$ by orthogonalizing only one vector of size $n$ at each iteration.
Also, $\mathbf{R}_{j+1}$  can be expanded in an easy way, if $u_{j+1}^{(p)} \notin$ span$\{Q_j\}$ then the blocks $R_{j+1}^{(i)}$, $i=1,\dots,d,$ can be written as
\begin{equation*}
R_{j+1}^{(i)}=\left[ \left.\begin{array}{c} R_j^{(i)} \\ 0_{1\times j}\end{array}\right| r_{j+1}^{(i)}\right],\quad i=1,\dots,d,
\end{equation*}
and, if $u_{j+1}^{(p)} \in$ span$\{Q_j\}$, then $R_{j+1}^{(i)}=\left[R_j^{(i)} \quad r_{j+1}^{(i)} \right]$, $i=1,\dots, d$. Based on these ideas, the authors of \cite{cork} developed CORK, splitting the method into two levels of orthogonalization: the first level is to expand $Q_j$ into $Q_{j+1}$ and the second level is to expand $\mathbf{R}_j$ into $\mathbf{R}_{j+1}$.
We can see a basic pseudocode for the CORK method in Algorithm \ref{CORK}, whose complete explanation can be found in \cite{cork}. For simplicity, we assume that breakdown does not occur in Algorithm \ref{CORK}, i.e., $h_{j+1,j} \ne 0$ for all $j$.

\begin{algorithm}[H]
\caption{Compact rational Krylov method (CORK)}\label{CORK}
\begin{algorithmic}
\REQUIRE $Q_1\in \C^{n\times r_1}$ and $\bfR_1\in \C^{dr_1\times 1}$ with $Q_{1}^*Q_1=I_{r_1}$ and $\bfR_1^*\bfR_1=1$, where $r_1\leq d$.
\ENSURE Approximate eigenpairs $(\lambda,\mathbf{x})$ associated to $\mathbf{A}-\lambda \mathbf{B}$, with $\mathbf{A}$, $\mathbf{B}$ as in \eqref{pencilpep}.
\FOR {$j=1,2,\dots$}
\STATE 1. Choose shift $\theta_j$.
\STATE First level of orthogonalization:
\STATE 2. Compute $\hat{u}^{(p)}$ by using the ULP decomposition in Theorem \ref{ULP} with $\mu = \theta_j$ (see \cite{cork} for details).
\STATE 3. Orthogonalize: $\tilde{q}=\hat{u}^{(p)}-Q_jQ_j^*\hat{u}^{(p)}$.
\STATE 4. If $\tilde{q} \ne 0$ then compute next vector: $q_{j+1}=\tilde{q} / \|\tilde{q}\|_2$ and $Q_{j+1}=[Q_j \quad q_{j+1}]$. Otherwise $Q_{j+1} = Q_j$.
\STATE Second level of orthogonalization:
\STATE 5. If $r_{j+1} > r_j$ then update matrices: $R_j^{(i)}=\left[ \begin{array}{c} R_j^{(i)} \\ 0_{1 \times j} \end{array}\right]$ for $i=1,\dots ,d$.
\STATE 6. Compute: $\hat{\mathbf r}$ by using the ULP decomposition in Theorem \ref{ULP} (see \cite{cork} for details).
\STATE 7. Compute: $\tilde{\mathbf{r}}=\hat{\mathbf{r}}-\bfR_jh_j$, where $h_j=\bfR_j^*\hat{\mathbf{r}}$.
\STATE 8. Next vector: $\mathbf{r}_{j+1}=\tilde{\mathbf r} / h_{j+1,j}$, where $h_{j+1,j}=\|\tilde{r}\|_2$ and $\mathbf{R}_{j+1}=[\mathbf{R}_j \quad \mathbf{r}_{j+1}]$.
\STATE 9. Compute eigenpairs: $(\lambda_i,t_i)$ of \eqref{solveHK} and test for convergence.
\ENDFOR
\STATE 10. Compute eigenvectors: $\mathbf{x}_i=(I_d \otimes Q_{j+1})\bfR_{j+1}\underbar{$H$}_jt_i$.
\end{algorithmic}
\end{algorithm}

From the discussion above, it is clear that CORK reduces significantly the storage requirements with respect to a direct application of the rational Krylov method to the $(nd) \times (nd)$ GEP $\mathbf{A} - \la \mathbf{B}$, since essentially CORK represents $\mathbf{U}_j$ in terms of $n (j+d)$ parameters and, in addition, $n (j+d)\approx n j$ for moderate values of $d$. Therefore, the memory cost of CORK is approximately the cost of any Krylov method applied to an $n\times n$ GEP. Moreover, it can be seen in \cite[Section 5.4]{cork} that the orthogonalization cost of CORK is essentially independent of $d$ for moderate values of $d$, and, so, much lower than the orthogonalization cost of a direct application of rational Krylov to $\mathbf{A} - \la \mathbf{B}$. With respect to the comparison of the costs of the shift-and-invert steps in CORK (included in step 2 of Algorithm \ref{CORK}) and in rational Krylov (step 2 in Algorithm \ref{rat.krylov}), we can say that in CORK the particular structure of the pencil $\mathbf{A} - \la \mathbf{B}$ together with the fact that only one block of the vector $\mathbf{\hat{u}}$ is needed allow us to perform this step very efficiently by essentially solving just one ``difficult'' $n\times n$ linear system (see \cite[Algorithm 2]{cork}). In contrast, in rational Krylov the whole vector $\hat{\mathbf{u}}$ must be computed and there is some extra cost with respect to CORK even in the case the structure of $\mathbf{A} - \la \mathbf{B}$ is taken into account for solving the linear system $(\mathbf{A} - \theta_j \mathbf{B}) \mathbf{\hat{u}} = \mathbf{B} \mathbf{u}_j$. On the other hand,  there is some overhead cost involved in step 2 of  Algorithm \ref{CORK}, since, in CORK, the actual vector $\mathbf{u}_j$ has to be constructed before solving the linear system associated to the shift-and-invert step. Fortunately, according to \eqref{blockU}, this computation can be arranged as the single matrix-matrix product $Q_j [r_j^{(1)} \, \cdots \, r_j^{(d)}]$, where $r_j^{(1)}, \ldots , r_j^{(d)}$ are the blocks of the last column of $\mathbf{R}_j$, which allows optimal efficiency and cache usage on modern computers (see \cite[p. 577]{Kressner}).

Inspired in CORK, we will develop in Section \ref{sect.newTOAR} the new algorithm R-CORK to solve large-scale and sparse rational eigenvalue problems by using a decomposition similar to \eqref{blockU} for the bases of the rational Krylov subspaces associated to the linearization \eqref{linearrat} of the REP and by working in the spirit of the two levels of orthogonalization originally introduced in TOAR \cite{charlaTOAR,toar}. We will see that R-CORK has memory and computational advantages similar to those discussed for CORK in the previous paragraph.

\section{A new method for solving large-scale and sparse rational eigenvalue problems} \label{sect.newTOAR}

\subsection{A compact decomposition for rational Krylov subspaces of $\cA- \la \cB$}\label{comp.decomp}

Consider the matrices $\mathcal{A}$ and $\mathcal{B}$ in \eqref{matAB} and the rational Krylov recurrence relation \eqref{rec.relation} which is valid for arbitrary pencils. Our goal is to particularize such relation to the matrices $\mathcal{A}$ and $\mathcal{B}$ in \eqref{matAB} in order to save memory and orthogonalization costs. For this purpose, we will partitionate $\UU_{j+1}$ conformably to $\mathcal{A}$ and $\mathcal{B}$ as follows
\begin{equation}\label{divU}
\UU_{j+1}=[\UU_j \quad \mathbf{u}_{j+1}]=\left[ \begin{array}{cc} U_j^{(1)} & u_{j+1}^{(1)} \\ U_{j}^{(2)} & u_{j+1}^{(2)} \\ \vdots & \vdots \\ U_{j}^{(d)} & u_{j+1}^{(d)} \\  V_{j} & v_{j+1} \end{array} \right]
\end{equation}
where $U_j^{(i)} \in \BC^{n\times j}$, $u_{j+1}^{(i)}\in\BC^n$, for $i=1,\dots,d$, $V_j \in \BC^{s\times j}$,  and $v_{j+1}\in \BC^s$. Next, following CORK for the first $d$ blocks, we define the matrix $Q_j\in \BC^{n\times r_j}$ such that the columns of $Q_j$ are orthonormal with
\begin{equation}\label{defQ}
\mbox{span}\{Q_j\}=\mbox{span}\{U_j^{(1)},U_j^{(2)},\dots,U_j^{(d)}\}
\end{equation}
and rank$(Q_j)=r_j$.
Using \eqref{defQ} we can express
\begin{equation}\label{def.decom}
U_j^{(i)} = Q_jR_j^{(i)}, \quad i=1,2,\dots,d,
\end{equation}
where $R_j^{(i)}\in\BC^{r_j\times j}$ for $i=1,2,\dots,d$. Then, by using \eqref{def.decom}, we have
\begin{equation}\label{blocks}
\UU_j=\left[\begin{array}{c} Q_jR_j^{(1)} \\ Q_jR_j^{(2)}  \\ \vdots \\ Q_jR_j^{(d)} \\  V_j \end{array}\right] =
\left[\begin{array}{ccccc} Q_j & & & & \\ & Q_j & & &  \\ & & \ddots & & \\ & & & Q_j &  \\
& & & & I_s \end{array}\right]
\left[\begin{array}{c} R_j^{(1)} \\ R_j^{(2)}  \\ \vdots \\ R_j^{(d)} \\  V_j \end{array}\right].
\end{equation}
By introducing the notation
\begin{equation}\label{defQ.R}
\mathbf{Q}_j := \left[\begin{array}{c|c} (I_d \otimes Q_j) & 0_{dn\times s}   \\ \hline 0_{s\times dr_j} & I_s \end{array}\right] \in \BC^{(dn+s)\times (dr_j+s)} \quad \mbox{and } \quad \mathbf{R}_j:= \left[\begin{array}{c} R_j^{(1)} \\ R_j^{(2)}  \\ \vdots \\ R_j^{(d)} \\  V_j \end{array}\right] \in \BC^{(dr_j+s)\times j},
\end{equation}
we have $\UU_j=\mathbf{Q}_j\bfR_j$. Note that the columns of $\UU_j$ and $\mathbf{Q}_j$ are orthonormal, so the matrix $\mathbf{R}_j$ has orthonormal columns too.
With this notation, we can rewrite \eqref{rec.relation} as the following compact rational Krylov recurrence relation
\begin{equation}\label{cork.relation}
\AA \mathbf{Q}_{j+1} \mathbf{R}_{j+1} \underbar{$H$}_j = \BB \mathbf{Q}_{j+1} \mathbf{R}_{j+1} \underbar{$K$}_j.
\end{equation}

In order to prove that, as in CORK, we need only one vector to expand $Q_j$ into $Q_{j+1}$ and that, as a consequence, $r_j$ is considerably smaller than $jd$, i.e., that $\mathbf{Q}_j\bfR_j$ is indeed a compact representation of $\UU_j$, we will prove first the following Lemmas \ref{ULP.REP} and \ref{lemma.system}. We emphasize the relationship between Lemma \ref{ULP.REP} and Theorem \ref{ULP}, but also the difference coming from the presence of the strictly proper part $E(C-\la D)^{-1}F^T$ of the rational matrix $R(\la)$, which motivates the definition of the rational matrix $\mathbf{A}(\la)$ in Lemma \ref{ULP.REP}. Apart from this difference, we have stated Lemma \ref{ULP.REP} in a completely analogous way to Theorem \ref{ULP}, with the purpose of stressing the relation with CORK, but note that the simple particular structures of $M$, $N$, and $\mathbf{B}$ inherited from \eqref{matAB}-\eqref{ABreduc} imply that in Lemma \ref{ULP.REP}
\begin{equation*}
m_0=0_{(d-1)\times 1} \quad \mbox{and} \quad n_0=-e_{d-1}=-\left[\begin{array}{c} 0 \\ \vdots \\ 0 \\ 1 \end{array}\right] \in \BC^{d-1},
\end{equation*}
that $M_1$ and $N_1$ are also very simple, and that $\mathbf{\bar{B}}_1$ has only one nonzero block. Therefore, the factors $\cL(\mu)$ and $\cU(\mu)$ in Lemma \ref{ULP.REP} are simpler than the general ones in Theorem \ref{ULP}. Note also that Lemma \ref{lemma.system} is related to \cite[Lemma 4.3]{cork}, although again the strictly proper part of the rational matrix introduces relevant differences.
\begin{lemma}\label{ULP.REP}
Consider a rational matrix
$$
R(\lambda) = P(\lambda)-E(C-\lambda D)^{-1}F^T \in \BC (\la)^{n \times n},
$$
where $P(\lambda)=\sum_{i=0}^d{\lambda^iP_i}$, $P_i\in \BC^{n\times n}$ for $i=0,\dots,d$, $E$, $F\in \BC^{n \times s}$, $C$, $D\in \BC^{s\times s}$, $D$ is nonsingular, and $E(C-\lambda D)^{-1}F^T$ is a minimal realization.  Define the rational matrix
\begin{equation*}
\mathbf{A}(\lambda)=\left[ \dfrac{P_{d-1} \quad P_{d-2} \quad \cdots \quad P_1 \quad P_0-E(C-\lambda D)^{-1}F^T}{M\otimes I_n} \right],
\end{equation*}
and the constant matrix
\begin{equation*}
\mathbf{B}=\left[ \dfrac{-P_d \quad 0_n \quad \cdots \quad 0_n \quad 0_n}{N\otimes I_n} \right],
\end{equation*}
with $M=-I_{(d-1)\times d}$ and $N=-[0_{(d-1)\times 1}\quad I_{d-1}]$.
Let $\cP\in\BC^{d\times d}$ be any matrix permutation such that $\cP$ moves the first column of a matrix to the last column, this is,
\begin{equation}\label{permP}
\cP=[\cP_1 \quad e_1], \quad \cP_1\in\BC^{d\times (d-1)}, \quad \cP_1^T\cP_1 = I_{d-1}.
\end{equation}
Then, for every $\mu\in\BC$ which is not a pole of $R(\mu)$, i.e., such that $(C-\mu D)$ is nonsingular, we can factorize $\mathbf{A}(\mu) -\mu \mathbf{B}$ as follows
\begin{equation}\label{ULP.rep}
\mathbf{A}(\mu)-\mu\mathbf{B}=\cU(\mu)\cL(\mu)(\cP\otimes I_n),
\end{equation}
where
\begin{eqnarray*}
\mathcal{L}(\mu) &=& \left[ \begin{array}{cc} R(\mu) & 0 \\ (m_0-\mu n_0)\otimes I_n & (M_1 -\mu N_1)\otimes I_n \end{array} \right], \\
\mathcal{U}(\mu) &=& \left[ \begin{array}{cc} I_n & (\mathbf{\bar{A}}_1-\mu \mathbf{\bar{B}}_1)((M_1-\mu N_1)^{-1}\otimes I_n) \\ 0 & I_{(d-1)n} \end{array} \right],
\end{eqnarray*}
with
\begin{equation*}
M=:[m_0 \quad M_1]\cP, \quad
N=:[n_0 \quad N_1]\cP,
\end{equation*}
and
\begin{eqnarray*}
[P_{d-1} \quad \cdots \quad P_1 \quad P_0-E(C-\mu D)^{-1}F^T] &=:& [P_0-E(C-\mu D)^{-1}F^T \quad \mathbf{\bar{A}}_1](\cP\otimes I_n), \\
{[}-P_{d} \quad 0_{n} \quad \cdots  \quad  0_n{]} &=:& {[} 0_n \quad \mathbf{\bar{B}}_1 {]}(\cP \otimes I_n).
\end{eqnarray*}
\end{lemma}
\begin{proof}
Observe first that the definitions of $M$, $N$ and $\cP$ imply trivially that $M_1-\mu N_1$ is nonsingular for every $\mu \in \BC$. By a direct matrix multiplication, we obtain
\begin{eqnarray*}
\cU(\mu)\cL(\mu) &=& \left[ \begin{array}{cc} R(\mu)+({\mathbf{\bar{A}}}_1-\mu{\mathbf{\bar{B}}}_1)(((M_1-\mu N_1)^{-1}(m_0-\mu n_0))\otimes I_n) & ({\mathbf{\bar{A}}}_1-\mu{\mathbf{\bar{B}}}_1)  \\
(m_0 - \mu n_0) \otimes I_n & (M_1-\mu N_1)\otimes I_n \end{array} \right] \\
&=& \left[ \dfrac{R(\mu)+({\mathbf{\bar{A}}}_1-\mu{\mathbf{\bar{B}}}_1)(((M_1-\mu N_1)^{-1}(m_0-\mu n_0))\otimes I_n) \quad ({\mathbf{\bar{A}}}_1-\mu{\mathbf{\bar{B}}}_1)}{(M-\mu N)\cP^T\otimes I_n} \right].
\end{eqnarray*}
Therefore, we only need to prove that
\begin{equation*}
R(\mu)+({\mathbf{\bar{A}}}_1-\mu{\mathbf{\bar{B}}}_1)(((M_1-\mu N_1)^{-1}(m_0-\mu n_0))\otimes I_n)= P_0-E(C-\mu D)^{-1}F^T,
\end{equation*}
which is equivalent to prove that
\begin{equation}\label{ULP.prove}
P(\mu)+({\mathbf{\bar{A}}}_1-\mu{\mathbf{\bar{B}}}_1)(((M_1-\mu N_1)^{-1}(m_0-\mu n_0))\otimes I_n)= P_0.
\end{equation}
The proof of \eqref{ULP.prove} is a very simple algebraic manipulation as a consequence of the extremely simple structures of $m_0$ and $n_0$, $M_1$ and $N_1$ in this case. Another proof comes from the observation that \eqref{ULP.prove} holds because it is proved for proving the ULP decomposition in Theorem \ref{ULP} (see \cite[pp. 823-824]{cork}).
\end{proof}
\begin{remark} Observe that Theorem \ref{ULP} involves the constant $\alpha=e_1^T\cP f(\mu)$, which is not present in Lemma \ref{ULP.REP}. The reason is that in Lemma \ref{ULP.REP}, this constant is equal to $1$ as a consequence of the structure of $\cP$ and \eqref{eigenv}.
\end{remark}
\begin{lemma}\label{lemma.system}
Let $\AA$ and $\BB$ be the matrices defined in \eqref{matAB}. Consider the linear system
\begin{equation}\label{lemma.sys}
(\AA-\mu \BB)\mathbf{x}=\BB \mathbf{w},
\end{equation}
where $\mathbf{x}=[{x}^{(1)^T}, {x}^{(2)^T},\cdots,{x}^{(d)^T},{y}^T]^T$ and $\mathbf{w}=[{w}^{(1)^T}, {w}^{(2)^T},\cdots,{w}^{(d)^T},{z}^T]^T$, the blocks $x^{(i)}, w^{(i)} \in \BC^{n}$, $i=1,2,\dots, d$, $y,z \in\BC^{s}$, and $\mu$ is not a pole of $R(\mu)$, i.e., $(C-\mu D)$ is nonsingular.
Then, the block $x^{(d)}$  of $\mathbf{x}$  can be computed by solving the following $n\times n$ linear system whose coefficient matrix is $R(\mu)$ in \eqref{newreprat}:
\begin{equation*}
R(\mu)x^{(d)}=-P_dw^{(1)}-E(C-\mu D)^{-1}Dz+(\mathbf{\bar{A}}_1-\mu \mathbf{\bar{B}}_1)((M_1-\mu N_1)^{-1}\otimes I_n)w^{(2,\dots,d)},
\end{equation*}
where the matrices introduced in Lemma \ref{ULP.REP} are used and $w^{(2,\ldots ,d)} = [w^{(2)^T}, \cdots, w^{(d)^T}]^T$. The remaining blocks $x^{(i)}$ for $i=1,\cdots,d-1$ of $\mathbf{x}$ can be obtained as linear combinations of $x^{(d)}$ and $w^{(i)}$, $i=2,\dots, d$. More precisely, if the permutation matrix $\cP$ in \eqref{permP} is expressed as
$\cP = \left[
\begin{array}{cc}
0 & 1 \\
\widetilde{\cP} & 0
\end{array} \right],
$
then $x^{(1,\ldots ,d-1)} = [x^{(1)^T}, \cdots, x^{(d-1)^T}]^T$ satisfies the linear system
\[
(\widetilde{\cP} \otimes I_n) \, x^{(1,\ldots ,d-1)} = - ((M_1 - \mu N_1)^{-1} \otimes I_n) \, (w^{(2,\ldots ,d)} + ((m_0 - \mu n_0) \otimes I_n) x^{(d)}) \, .
\]
In addition, $y$ can be computed by solving the $s\times s$ linear system
\begin{equation*}
(C-\mu D)y=Dz-F^Tx^{(d)}.
\end{equation*}
\end{lemma}

\begin{proof}
Rewrite the matrix pencil \eqref{linearrat} as in \eqref{ABreduc}
\begin{equation*}
\AA-\mu \BB= \left[ \begin{array}{c|c} \bf{A}-\mu\bf{B} & e_1\otimes E \\ \hline
e_d^T\otimes F^T & C - \mu D \end{array}\right]
\end{equation*}
with
\begin{equation*}
\bfA=\left[\dfrac{P_{d-1} \quad P_{d-2} \quad \cdots \quad P_{0}}{M\otimes I_n}\right], \quad
\bfB=\left[\dfrac{-P_{d} \quad 0_{n} \quad \cdots \quad 0_{n}}{N\otimes I_n}\right],
\end{equation*}
\begin{equation*}
M=-I_{(d-1)\times d}, \quad \mbox{and} \quad N=-[0_{(d-1)\times 1}\quad I_{(d-1)}].
\end{equation*}
Then, we can solve the system \eqref{lemma.sys} by solving
\begin{eqnarray}
(\mathbf{A} -\mu \mathbf{B})x^{(1,2,\dots,d)}+(e_1\otimes E)y&=&\mathbf{B} w^{(1,2,\dots,d)}, \label{eq.1} \\
(e_d^T\otimes F^T)x^{(1,2,\dots,d)} + (C-\mu D)y &=& Dz, \label{eq.2}
\end{eqnarray}
where $x^{(1,2,\dots,d)}=[{x}^{(1)^T}, {x}^{(2)^T},\cdots,{x}^{(d)^T}]^T$ and $w^{(1,2,\dots,d)}=[{w}^{(1)^T}, {w}^{(2)^T},\cdots,{w}^{(d)^T}]^T$. The second equation is the equation for $y$ in the statement.
By replacing $y=(C-\mu D)^{-1}(Dz-F^Tx^{(d)})$ from \eqref{eq.2} in \eqref{eq.1}, and using the notation of Lemma \ref{ULP.REP}, we obtain
\begin{eqnarray*}
(\mathbf{A}-\mu \mathbf{B})x^{(1,2,\dots,d)}-(e_1\otimes E)(C-\mu D)^{-1}F^Tx^{(d)}&=&\mathbf{B}w^{(1,2,\dots,d)}-(e_1\otimes E)(C-\mu D)^{-1}Dz, \\
\left[\dfrac{P_{d-1}+\mu P_d \quad P_{d-2} \quad \cdots \quad P_0-E(C-\mu D)^{-1}F^T}{(M-\mu N)\otimes I_n}\right] x^{(1,2,\dots,d)}&=&
-\left[\begin{array}{c} P_dw^{(1)} +E(C-\mu D)^{-1}Dz\\ \hline w^{(2,\dots,d)}\end{array}\right], \\
(\mathbf{A}(\mu)-\mu \mathbf{B})x^{(1,2,\dots,d)}&=& -\left[\begin{array}{c} P_dw^{(1)} +E(C-\mu D)^{-1}Dz\\ \hline w^{(2,\dots,d)}\end{array}\right].
\end{eqnarray*}
By combining the factorization \eqref{ULP.rep} in Lemma \ref{ULP.REP} and the equation above, it is immediate to see that the blocks $x^{(i)}$ for $i=1,\dots,d-1$ of $\mathbf{x}$ are linear combinations of $x^{(d)}$ and the blocks $w^{(i)}$, $i=2,\dots, d$. In addition, some elementary matrix manipulations with the matrices $\mathcal{U}(\mu)$ and $\mathcal{L}(\mu)$ in \eqref{ULP.rep} and the structure of the permutation matrix $\cP$ lead to the equations for $x^{(1,\ldots , d-1)}$ and $x^{(d)}$ in the statement. This finishes the proof.
\end{proof}
\begin{remark}\label{rem.P}
Since the matrices $P_i$ in the matrix $\mathcal{A}$ in \eqref{matAB} appear in decreasing index order, it is natural to choose in Lemmas \ref{ULP.REP} and \ref{lemma.system} the permutation $\cP$ as follows
\begin{equation}\label{permPpart}
\mathcal{P}=\left[ \begin{array}{ccc}    &  & 1 \\   & \iddots  &  \\ 1  &  &  \end{array}\right].
\end{equation}
\end{remark}

As announced before, Lemma \ref{lemma.system} is the key result that allows us to prove trough Theorems \ref{th.Q} and \ref{th.rjRCORK} that only one vector is needed to expand $Q_j$ into $Q_{j+1}$ and, so, that the representation \eqref{blocks} for $\mathbf{U}_j$ is indeed compact. Moreover, the equations for $x^{(d)}$, $x^{(1,\ldots,d-1)}$, and $y$ deduced in Lemma \ref{lemma.system} lead to the efficient Algorithm \ref{solver.system} for solving the linear system \eqref{lemma.sys}, which is fundamental for performing efficiently the shift-and-invert step in the R-CORK method developed in Section \ref{RCORK-subsect}. Observe that in Algorithm \ref{solver.system} a notation similar to that in Lemma \ref{lemma.system} is used and that $t^{(1,\ldots,d)} := (\cP \otimes I_n) x^{(1,\ldots,d)}$ is computed first, performing later the inverse permutation for getting $x^{(1,\ldots,d)}$.

\begin{algorithm}[H]
\caption{Solver for the linear system $(\AA -\mu \BB) \mathbf{x}=\BB {\mathbf{w}}$, with $\cA$ and $\cB$ as in \eqref{matAB}}\label{solver.system}
\begin{algorithmic}
\REQUIRE $\AA$, $\BB \in \BC^{(nd+s)\times (nd+s)}$ as in \eqref{matAB}, $\mu \in \BC$ such that $(C-\mu D)^{-1}$ exists and $\mathbf{w} \in \C^{nd+s}$.
\ENSURE The solution $\mathbf{x}$ of the linear system.
\STATE 1. Compute $t=\mathbf{B}w^{(1,2,\dots,d)}-(e_1\otimes E)(C-\mu D)^{-1}Dz$ as $t=-\left[ \frac{P_dw^{(1)}+E(C-\mu D)^{-1}Dz}{w^{(2,\dots,d)}}\right]$.
\STATE Solve the block upper triangular system associated to $\cU(\mu)$ in \eqref{ULP.rep}:
\STATE 2. $t^{(1)}=t^{(1)}-(\mathbf{\bar{A}}_1-\mu \mathbf{\bar{B}}_1)((M_1-\mu N_1)^{-1}\otimes I_n)t^{(2,\dots,d)}$.
\STATE Solve the block lower triangular system associated to $\cL(\mu)$ in \eqref{ULP.rep}:
\STATE 3. $t^{(1)}=(R(\mu))^{-1}t^{(1)}$.
\STATE 4. $t^{(2,\dots,d)}=((M_1-\mu N_1)^{-1}\otimes I_n)(t^{(2,\dots,d)}-((m_0-\mu n_0)\otimes I_n)t^{(1)})$.
\STATE Permute the blocks of $t^{(1,2,\dots,d)}$:
\STATE 5. $x^{(1,2,\dots, d)}=(\mathcal{P}^T\otimes I_n)t^{(1,2,\dots,d)}$.
\STATE Compute the block $y$ of ${\mathbf{x}}$
\STATE 6. $y=(C-\mu D)^{-1}(Dz-F^Tx^{(d)})$.
\end{algorithmic}
\end{algorithm}
\begin{remark} \label{rem.costlinsolver}
The multiplications by inverses in Algorithm \ref{solver.system} have to be understood, in principle, as solutions of linear systems and the key observation on Algorithm \ref{solver.system} is that all the involved linear systems have sizes smaller than the size $(n d +s) \times (n d +s)$ of $(\cA - \mu \cB)$ as we discuss in this remark. The only linear system which is always large is the one in line 3 involving $R(\mu)$ which has the size $n\times n$ of the original REP. Clearly, solving the system in line 3 requires to construct $R(\mu)$ (analogously to \cite{Kressner,cork}), which in the case $(C-\mu D)$ is large and complicated might be performed more efficiently trough \eqref{formrat} than through \eqref{newreprat}, though this depends on each particular problem. However, we emphasize once again that the matrix $(C-\mu D)\in \BC^{s\times s}$ is in many applications \cite{visco,SuBai} very small, since $s \ll n$, and has in addition a very simple structure, which imply that it is often possible just to compute $(C-\mu D)^{-1}$ and to perform the corresponding matrix multiplications to construct $R(\mu)$ through \eqref{newreprat}. These comments on the size $s \ll n$ also apply to the linear systems involving $(C-\mu D)\in \BC^{s\times s}$ in lines 1 and 6 which are very often in practice very small. Finally, the linear systems involving $(M_1-\mu N_1)\otimes I_n$ have size $(d-1) n \times (d-1)n$ and look very large, but they are block linear systems very easy to solve with cost $2n(d-2)$ flops by using simple two term recurrence relations. For instance, if $\cP$ is the permutation in \eqref{permPpart} then
$$
M_1-\mu N_1=\left[ \begin{array}{cccc} & & \mu & -1 \\  & \iddots & \iddots & \\ \mu & \iddots & & \\ -1 &  &  & \end{array}\right] \in \BC^{(d-1)\times (d-1)},
$$
and the solution of $((M_1 - \mu N_1) \otimes I_n) \mathbf{x} = \mathbf{b}$ partitioning the vectors in $(d-1)$ blocks of size $n \times 1$ can be obtained as
$x^{(1)} = -b^{(d-1)}$ and $x^{(i)} = \mu \, x^{(i-1)} - b^{(d-i)}$ for $i=2,3, \ldots,d-1$.
\end{remark}
The following theorems are similar to results obtained in \cite[Theorems 4.4 and 4.5]{cork}.
\begin{theorem}\label{th.Q}
Let $Q_j$ be defined as in \eqref{defQ}. Then,
\begin{equation}
\mbox{span}\{Q_{j+1}\} = \mbox{span}\{Q_j,u_{j+1}^{(d)}\}.
\end{equation}
\end{theorem}
\begin{proof}
The proof is immediate from definition \eqref{defQ} and Lemma \ref{lemma.system} with $\mu=\theta_j$ and $\mathbf{w}=\mathbf{u}_j$ (see proof of \cite[Theorem 4.4]{cork}).
\end{proof}
\begin{theorem} \label{th.rjRCORK}
Let $Q_j$ be defined as in \eqref{defQ}. Then
\begin{equation}\label{ineq.r}
r_j<d+j.
\end{equation}
\end{theorem}
\begin{proof}
We will prove this theorem by induction. From the definition of $Q_j$ in \eqref{defQ}, we have that
\begin{equation*}
\mbox{span}\{Q_1\}=\mbox{span}\{u_1^{(1)}, u_1^{(2)},\dots,u_1^{(d)}\},
\end{equation*}
so $r_1\leq d$. Assuming that the inequality \eqref{ineq.r} is satisfied until $j-1$, then we have by Theorem \ref{th.Q} that $r_{j}\leq r_{j-1} +1 < d+j$.
\end{proof}
By considering the inequality \eqref{ineq.r} and the fact that $r_j$ increases at most $1$ in each iteration,  we will show the possible structures of the expansion of the first $d$ blocks of the matrix $\mathbf{R}_j$ defined in \eqref{defQ.R}.
\begin{lemma}\label{expand.R}
Let $\mathbf{R}_j\in \BC^{(dr_j+s)\times j}$ be defined as in \eqref{defQ.R}. Then, the first $d$ blocks of the matrix $\mathbf{R}_{j+1}\in \BC^{(dr_{j+1}+s)\times (j+1)}$ can take the following forms:
\begin{itemize}
\item if $r_{j+1}>r_j$
\begin{equation*}
R_{j+1}^{(i)}=\left[ \begin{array}{c} R_j^{(i)} \\ 0_{1\times j} \end{array} r_{j+1}^{(i)} \right], \quad i=1,2,\dots,d,
\end{equation*}
where $r_{j+1}^{(i)}\in \BC^{r_{j+1}}$, or
\item if $r_{j+1}=r_j$
\begin{equation*}
R_{j+1}^{(i)}=\left[ R_j^{(i)}\quad r_{j+1}^{(i)} \right], \quad i=1,2,\dots,d,
\end{equation*}
with $r_{j+1}^{(i)}\in \BC^{r_{j+1}}$.
\end{itemize}
\end{lemma}


\subsection{The R-CORK method}\label{RCORK-subsect}
In this section, we will introduce the method to solve large-scale and sparse rational eigenvalue problems based on the compact representation presented in Section \ref{comp.decomp} of the orthonormal bases of the rational Krylov subspaces of the linearization $\cA - \la \cB$ in \eqref{linearrat}. First, we consider an initial vector $\mathbf{u}_1 \in \BC^{nd+s}$ with $\|\mathbf{u}_1\|_2=1$ partitioned as in \eqref{divU} and then, we express that vector in a compact form:
\begin{equation*}
\mathbf{u}_1=\left[ \begin{array}{c} u_1^{(1)} \\ \vdots \\ u_1^{(d)} \\  v_1 \end{array} \right]= \left[ \begin{array}{c} Q_1R_1^{(1)} \\ \vdots \\ Q_1R_1^{(d)} \\  v_1 \end{array} \right]
\end{equation*}
where $Q_1 \in \BC^{n\times r_1}$ has orthonormal columns such that
\begin{equation*}
\mbox{span}\{Q_1\}=\mbox{span}\left\{u_1^{(1)},\cdots,u_1^{(d)}\right\}, \quad r_1=\mbox{rank}([u_1^{(1)}\cdots u_1^{(d)}]).
\end{equation*}
Observe that $r_1=1$ if and only if $\mathbf{u}_1$ is chosen to have collinear nonzero blocks $u_1^{(1)},\dots,u_1^{(d)}$. Now, taking into account the definition of $\mathbf{R}_j$ in \eqref{defQ.R}, after $j$ steps we want to expand $Q_j$ into $Q_{j+1}$ and $\mathbf{R}_j$ into $\mathbf{R}_{j+1}$, which results in the so-called two levels of orthogonalization.

{\bf First level of orthogonalization.} In Theorem \ref{th.Q} we have proved that we need to orthogonalize $u_{j+1}^{(d)}$ with respect to $Q_j$  to obtain the last orthonormal column of $Q_{j+1}$. In addition, it can be easily seen that
\begin{equation}\label{expandQ}
\mbox{span}\{Q_{j+1}\} = \mbox{span}\{Q_j,u_{j+1}^{(d)}\}=\mbox{span}\{Q_j,\hat{u}^{(d)}\},
\end{equation}
where $\hat{u}^{(d)}$ is the $d$-th block of size $n$ of the vector $\mathbf{\hat{u}}$ obtained by applying the shift-and-invert step to $\mathbf{u}_j$ (step 2 in Algorithm \ref{rat.krylov}) when $\mathbf{\hat{u}}$ is partitioned as in \eqref{divU}. Therefore, we only need to compute the block $\hat{u}^{(d)}$ of $\mathbf{\hat{u}}$ to compute $Q_{j+1}$. Thus, we can run Algorithm \ref{solver.system} with $\mathbf{w} = \mathbf{u}_j$ and $\mu = \theta_j$ until step 3, saving the resulting vector $t^{(1)} = \hat{u}^{(d)}$. It is important to observe that the first $d$ blocks of $\mathbf{u}_j$ have to be constructed, since the variables in R-CORK are $Q_j$ and $\mathbf{R}_j$, and $\mathbf{u}_j$ is not stored. As in CORK, they are computed as the single matrix-matrix product $Q_j \, [r_j^{(1)} \cdots r_j^{(d)}]$, where $r_j^{(1)} ,  \ldots , r_j^{(d)}$ are the first $d$ blocks of the last column $\mathbf{r}_j$ of $\mathbf{R}_j$, which is a very efficient computation in terms of cache utilisation on modern computers. Once $\hat{u}^{(d)}$ is available, by \eqref{expandQ} we can decompose
\begin{equation}\label{decomp.u}
\hat{u}^{(d)}=Q_jx_j+\alpha_j q_{j+1},
\end{equation}
where $q_{j+1}$ is a unit vector orthogonal to $Q_j$ and $x_j = Q_j^* \hat{u}^{(d)}$. Observe also that since $\hat{u}^{(d)}$ has been already computed, we can compute the last $s$ entries of $\mathbf{\hat{u}}$, denoted by $\hat{v}$, from step 6 in Algorithm \ref{solver.system} without the need of performing steps 4 and 5. The vector $\hat{v}$ will be used in the second level of orthogonalization.
Now, if $\hat{u}^{(d)}$ does not lie in the subspace spanned by the columns of $Q_j$, i.e., if $\hat{u}^{(d)}-Q_jx_j \ne 0$, we can expand $Q_j$ into $Q_{j+1}$ as follows:
\begin{equation*}
Q_{j+1}=[Q_j \quad q_{j+1}], \quad r_{j+1}=r_j+1.
\end{equation*}
On the other hand, if $\hat{u}^{(d)}$ lies in the subspace spanned by the columns of $Q_j$, we have $Q_{j+1}=Q_j$ and $r_{j+1}=r_j$.
We summarize the first level of orthogonalization in Algorithm \ref{first.level}. In step 2, if it is necessary, we can reorthogonalize $\tilde{q}$ to ensure orthogonality. In fact, in our  MATLAB code, we perform the classical Gram-Schmidt method twice.
\begin{algorithm}[H]
\caption{First level of orthogonalization in R-CORK}\label{first.level}
\begin{algorithmic}
\REQUIRE The matrix $Q_j\in\BC^{n\times r_j}$ and the vector $\hat{u}^{(d)}\in\BC^n$ (the $d$-th block of $\mathbf{\hat{u}}=(\AA - \theta_j \BB)^{-1}\BB\mathbf{u}_j$).
\ENSURE The matrix $Q_{j+1}\in\BC^{n\times r_{j+1}}$, the vector $x_j$, and the scalar $\alpha_j$.
\STATE Expanding $Q_j$ into $Q_{j+1}$.
\STATE 1. $x_j=Q_j^*\hat{u}^{(d)}$.
\STATE 2. $\tilde{q}=\hat{u}^{(d)}-Q_jx_j$.
\STATE 3. $\alpha_j=\|\tilde{q}\|_2$.
\IF {$\alpha_j \neq 0$}
\STATE 4a. $Q_{j+1}=[Q_j \quad \tilde{q}/ \alpha_j$].
\STATE 5a. $r_{j+1}=r_j+1$.
\ELSE
\STATE 4b. $Q_{j+1}=Q_j$.
\STATE 5b. $r_{j+1}=r_j$.
\ENDIF
\end{algorithmic}
\end{algorithm}

{\bf Second level of orthogonalization.} In Algorithm \ref{rat.krylov}, after choosing the shift and performing the shift-and-invert step, we need to compute the entries of the $j$-th column of $\underline{H}_j$ in step 3. Let us see how to do it efficiently in R-CORK. By using the compact representation of $\UU_j$ in \eqref{blocks} - \eqref{defQ.R}, we have
\begin{eqnarray}
h_j&=&\UU_j^*\mathbf{\hat{u}}, \nonumber \\
&=& (R_j^{(1)})^*Q_j^*\hat{u}^{(1)}+\cdots+(R_j^{(d)})^*Q_j^*\hat{u}^{(d)}+V_j^*\hat{v}, \label{hj}
\end{eqnarray}
where $\mathbf{\hat{u}}$ has been partitioned in an analogous way to \eqref{divU}.
Since $(\AA-\theta_j \BB)\mathbf{\hat{u}}=\BB\mathbf{u}_j$, and $\AA$ and $\BB$ have the structures in \eqref{matAB}, we obtain the following relation between the blocks of size $n$ of $\mathbf{\hat{u}}$ and the blocks of size $n$ of $\mathbf{u}_j$
\begin{equation}\label{rec.uhat}
\hat{u}^{(i-1)}=\theta_j\hat{u}^{(i)}+u^{(i)}_j, \quad \mbox{for }i=d,d-1,\dots,2.
\end{equation}
Motivated by \eqref{rec.uhat}, we consider the vectors $x_j\in\BC^{r_j}$ obtained in step 1 in Algorithm \ref{first.level} and $\hat{v}\in\BC^{s}$ obtained in step 6 in Algorithm \ref{solver.system} with $x^{(d)}=\hat{u}^{(d)}$ and $\mu = \theta_j$, and a vector $\mathbf{\hat{p}}\in\BC^{dr_j+s}$ partitioned as follows
\begin{equation} \label{eq.phatdef}
\mathbf{\hat{p}}=\left[ \begin{array}{c} \hat{p}^{(1)} \\ \vdots \\ \hat{p}^{(d)} \\  \hat{v} \end{array} \right], \quad \hat{p}^{(i)}\in\BC^{r_j}, \quad i=1,\dots,d,
\end{equation}
with the blocks defined by the recurrence relation
\begin{eqnarray} \label{eq.phatrecurrence}
\hat{p}^{(d)}&=&x_j, \nonumber \\
\hat{p}^{(i-1)}&=&\theta_j \hat{p}^{(i)} + r_j^{(i)}, \quad i=d,d-1,\dots,2, \label{phat}
\end{eqnarray}
where $r_{j}^{(i)}$ represents the $j$-th column of the block $R_j^{(i)}$ in \eqref{blocks}.  If $\alpha_j\neq 0$ in step 3 in Algorithm \ref{first.level}, by using $\mathbf{\hat{p}}$, the decomposition \eqref{decomp.u} and the recurrence relation \eqref{rec.uhat}, the vectors $\hat{u}^{(i)}$, $i=1,\dots,d$, corresponding to the partition of $\mathbf{\hat{u}}$ as in \eqref{divU} can be represented as follows
\begin{equation}\label{def.uhat}
\hat{u}^{(i)}=Q_{j+1}\left[ \begin{array}{c} \hat{p}^{(i)} \\ \theta_j^{d-i}\alpha_j\end{array} \right], \quad i=1, \dots, d,
\end{equation}
whereas that if $\alpha_j=0$, we can represent the blocks $\hat{u}^{(i)}$, $i=1,\dots,d$, as follows
\begin{equation}\label{def.uhat0}
\hat{u}^{(i)}=Q_j\hat{p}^{(i)}.
\end{equation}
Then, by using either \eqref{def.uhat} or \eqref{def.uhat0} (depending on the value of $\alpha_j$) in \eqref{hj} and recalling that the columns of $Q_{j+1}$ are orthonormal, we have
\begin{equation}\label{def.hj}
h_j=\left[\begin{array}{c} R_{j}^{(1)} \\ \vdots \\ R_{j}^{(d)}\\  V_j\end{array}\right]^{*} \left[\begin{array}{c} \hat{p}^{(1)} \\ \vdots \\ \hat{p}^{(d)}\\  \hat{v}\end{array}\right]=\mathbf{R}_j^*\mathbf{\hat{p}}.
\end{equation}
Thus, after computing $\mathbf{\hat{p}}$ with the recurrence relation \eqref{eq.phatrecurrence}, we can compute $h_j$ by performing a matrix-vector multiplication of size $dr_j+s$, which, according to \eqref{ineq.r}, is much smaller than $dn + s$ in large-scale problems and, even more, much smaller than $n$ whenever $s\ll n$ as often happens in applications \cite{visco,SuBai}.

Next, in step 4 of Algorithm \ref{rat.krylov}, we need to compute the vector $\mathbf{\tilde{u}}$, which means that in R-CORK we need its compact representation. By using the compact representation of $\mathbf{U}_j$ and \eqref{def.uhat}, we have, if $\alpha_j\neq 0$,
\begin{eqnarray*}
\mathbf{\tilde{u}}&=&\mathbf{\hat{u}}-\UU_jh_j,  \\
 &=& \left[ \begin{array}{c} Q_{j+1}\left[ \begin{array}{c} \hat{p}^{(1)} \\ \theta_j^{d-1}\alpha_j \end{array} \right]  \\ \vdots \\ Q_{j+1} \left[ \begin{array}{c} \hat{p}^{(d)} \\ \alpha_j \end{array} \right] \\  \hat{v}
 \end{array} \right] - \left[\begin{array}{c} Q_jR_j^{(1)} \\ \vdots \\ Q_jR_j^{(d)} \\ V_j\end{array} \right]h_j,
 \\
 &=& \left[ \begin{array}{cccc} Q_{j+1} & & & \\ & \ddots & & \\ & & Q_{j+1} & \\ & & & I_s \end{array}\right]
 \left[ \begin{array}{c} \left[\begin{array}{c}\hat{p}^{(1)} - R_j^{(1)}h_j \\ \theta_j^{d-1}\alpha_j \end{array}\right] \\ \vdots  \\\left[\begin{array}{c} \hat{p}^{(d)}-R_j^{(d)}h_j \\ \alpha_j \end{array}\right]\\ \hat{v}-V_jh_j \end{array} \right],
\end{eqnarray*}
and, in a similar way, if $\alpha_j = 0$ we obtain
\begin{equation*}
\mathbf{\tilde{u}}=\left[ \begin{array}{cccc} Q_{j} & & & \\ & \ddots & & \\ & & Q_{j} & \\ & & & I_s \end{array}\right]
 \left[ \begin{array}{c} \hat{p}^{(1)} - R_j^{(1)}h_j  \\ \vdots  \\ \hat{p}^{(d)}-R_j^{(d)}h_j \\ \hat{v}-V_jh_j \end{array} \right].
\end{equation*}
Defining
\begin{equation}\label{ptilde}
\mathbf{\tilde{p}}:=\left[\begin{array}{c} \tilde{p}^{(1)} \\ \vdots \\ \tilde{p}^{(d)} \\ \tilde{v} \end{array} \right], \quad
\tilde{p}^{(i)}:=\hat{p}^{(i)}-R_j^{(i)}h_j \in\BC^{r_j}, \,\, i=1,\dots,d, \quad
\tilde{v} := \hat{v}-V_jh_j \in \BC^{s},
\end{equation}
and taking into account that the columns of $Q_{j+1}$ and $Q_j$ are orthonormal, we can express the step 5 in Algorithm \ref{rat.krylov} as follows: if $\alpha_j \ne 0$, then
\begin{equation}\label{betaj}
h_{j+1,j}=\| \mathbf{\tilde{u}} \|_2= \left\| \left[ \begin{array}{c} \tilde{p}^{(1)} \\ \theta_j^{d-1} \alpha_j \\ \vdots \\  \tilde{p}^{(d)} \\  \alpha_j \\ \tilde{v} \end{array} \right] \right\|_2 \, \mbox{and } \mathbf{u}_{j+1}=\mathbf{Q}_{j+1}\cdot \dfrac{1}{h_{j+1,j}}\left[ \begin{array}{c} \tilde{p}^{(1)} \\ \theta_j^{d-1} \alpha_j \\ \vdots \\  \tilde{p}^{(d)} \\  \alpha_j \\ \tilde{v}  \end{array} \right],
\end{equation}
where the notation in \eqref{defQ.R} is used, while if $\alpha_j = 0$ we proceed as in \eqref{betaj} by removing all the entries involving $\alpha_j$ and with $\mathbf{Q}_{j+1} = \mathbf{Q}_{j}$.
From the previous equations, we can conclude that the first $d$ blocks of size $r_{j+1}$ of the last column of $\mathbf{R}_{j+1}$ in \eqref{defQ.R} are given if $\alpha_j \neq 0$ by
\begin{equation}\label{last.R}
r_{j+1}^{(i)}=\dfrac{1}{h_{j+1,j}}\left[\begin{array}{c} \tilde{p}^{(i)} \\ \theta_j^{d-i} \alpha_j  \end{array}\right], \quad i=1,\dots, d,
\end{equation}
and if $\alpha_j =0$ by
\begin{equation} \label{last.R2}
r_{j+1}^{(i)}= \dfrac{1}{h_{j+1,j}} \, \tilde{p}^{(i)}, \quad i=1,\dots, d.
\end{equation}
In addition, the last block of size $s$ of the last column of $\mathbf{R}_{j+1}$ is
\begin{equation}\label{v.j+1}
v_{j+1}=\dfrac{1}{h_{j+1,j}}\tilde{v}.
\end{equation}
Since $\mathbf{R}_j$ has orthonormal columns, from \eqref{def.hj} and the definitions in \eqref{eq.phatdef}-\eqref{phat} and \eqref{ptilde}, we have that $h_j$ and $\mathbf{\tilde{p}}$ satisfy
\begin{equation}
\mathbf{\tilde{p}}=\mathbf{\hat{p}}-\mathbf{R}_jh_j,
\end{equation}
where $\mathbf{\tilde{p}}$ is orthogonal to $\mathbf{R}_j$. This process is the Gram-Schmidt process without the normalization step, and it is summarized in Algorithm \ref{second.level}.
\begin{algorithm}
\caption{Second level of orthogonalization in R-CORK}\label{second.level}
\begin{algorithmic}
\REQUIRE The matrix $\mathbf{R}_j \in \BC^{(dr_j+s) \times j}$ and the vector $\mathbf{\hat{p}}\in \BC^{dr_j+s}$ from \eqref{eq.phatdef}-\eqref{phat}.
\ENSURE Vectors $h_j\in \BC^{j}$ and $\mathbf{\tilde{p}}\in\BC^{dr_j+s}$.
\STATE 1. $h_j=\mathbf{R}_j^*\mathbf{\hat{p}}$.
\STATE 2. $\mathbf{\tilde{p}}=\mathbf{\hat{p}}-\mathbf{R}_jh_j$.
\end{algorithmic}
\end{algorithm}
\begin{remark} In order to improve orthogonality, a reorthogonalization method can be included in Algorithm \ref{second.level}. In our MATLAB code, we use the classical Gram-Schmidt process twice.
\end{remark}
The whole procedure of this new method to solve large-scale and sparse rational eigenvalue problems requires the use of the two levels of orthogonalization described in this section, the first level to expand $Q_j$ into $Q_{j+1}$ and the second level to expand $\mathbf{R}_j$ into $\mathbf{R}_{j+1}$. The complete R-CORK method is summarized in Algorithm \ref{RCORK}. Note that R-CORK has as inputs the matrix $Q_1$ and the vector $\mathbf{R}_1$, which have to be computed. As in CORK \cite[p. 830]{cork}, there are two possible ways of computing these inputs: either starting with a random vector $\mathbf{u}_{1} \in \BC^{nd +s}$ and using an economy-size QR factorization, or emulating the structure of the eigenvectors \eqref{eigenv} of the linearization in \eqref{linearrat}. The details are very similar to the ones in \cite[p. 830]{cork} and are omitted. Recall in Algorithm \ref{RCORK} that $\mathbf{r}_{j}$ denotes the last column of the matrix $\mathbf{R}_{j}$ in \eqref{defQ.R}.

\begin{algorithm}[H]
\caption{Compact rational Krylov method for REP (R-CORK)}\label{RCORK}
\begin{algorithmic}
\REQUIRE ${Q}_1\in \BC^{n\times r_1}$, $\bfR_1\in \C^{(dr_1+s) \times 1}$ with $Q_{1}^*Q_1=I_{r_1}$ and $\bfR_1^*\bfR_1=1$.
\ENSURE Approximate eigenpairs $(\lambda,\mathbf{x})$ of $\AA-\lambda \BB$ with $\AA$ and $\BB$ as in \eqref{matAB}.\\
\FOR {$j=1,2,\dots$}
\STATE 1. Choose shift $\theta_j$.
\STATE 2. Compute $\mathbf{u}_{j}= \mathbf{Q}_{j}\mathbf{r}_{j}$, obtaining the first $d$ blocks as matrix-matrix product $Q_{j} \, [r_{j}^{(1)} \cdots r_{j}^{(d)}]$.
\STATE 3. Compute $\hat{u}^{(d)}$ by using Algorithm \ref{solver.system} until step 3 applied to $\mathbf{w}=\mathbf{u}_j$ and $\mu = \theta_j$.
\STATE 4. Compute $\hat{v}$ from step 6 in Algorithm \ref{solver.system}.
\STATE First level of orthogonalization
\STATE 5. Run Algorithm \ref{first.level} obtaining $Q_{j+1}$, the scalar $\alpha_j$ and the vector $x_j$.
\STATE Second level of orthogonalization:
\STATE 6. Compute $\mathbf{\hat{p}}$ in \eqref{eq.phatdef} via the recurrence relation in \eqref{phat}.
\STATE 7. Run Algorithm \ref{second.level} obtaining $\mathbf{\tilde{p}}$ and $h_j$.
\STATE 8. Compute $h_{j+1,j}$ and $\mathbf{r}_{j+1}$ using \eqref{betaj}-\eqref{last.R}-\eqref{last.R2}-\eqref{v.j+1} and get $\mathbf{R}_{j+1}$ with Lemma \ref{expand.R}.
\STATE 9. Compute eigenpairs: $(\lambda_i,t_i)$ of \eqref{solveHK} and test for convergence.
\ENDFOR
\STATE 10. Compute eigenvectors: $\mathbf{x}_i= \mathbf{Q}_{j+1}\bfR_{j+1}\underbar{$H$}_jt_i$.
\end{algorithmic}
\end{algorithm}


\subsection{Memory and computational costs}
In this section, we discuss the memory and the computational costs of R-CORK
and compare these costs with those of the classical rational Krylov (RK) method, i.e., Algorithm \ref{rat.krylov}, applied directly to the linearization $\cA - \la \cB$ of the REP in \eqref{linearrat}. In order to simplify the results we will take $r_j = j + d$ in R-CORK, which is the upper bound in Theorem \ref{th.rjRCORK} and that essentially corresponds to start the R-CORK iteration with $Q_1 \in \BC^{n \times d}$ ($r_1 = d$) or, equivalently, with a random initial vector $\mathbf{u}_1$ whose first $d$ blocks in the partition \eqref{divU} are linearly independent. If the first $d$ blocks of $\mathbf{u}_1$ are taken to be collinear, then one can take $r_j = j$ and to improve even more the costs of R-CORK. In addition, note that
we estimate the costs for any value of $s$, where $s \times s$ is the size of the lower-right block $C - \la D$ of $\cA - \la \cB$ appearing in the strictly proper part of the REP \eqref{newreprat}. In this way, it will be seen that even if $s \approx n$, R-CORK has considerable advantages with respect to RK in terms of memory and computational costs. However, we emphasize that such advantages are still much more relevant when $s\ll n$, as happens very often in applications \cite{visco,SuBai}.

For the memory costs, after $j$ iterations R-CORK stores $Q_j \in \BC^{n \times r_j}$ and $\mathbf{R}_j \in \BC^{(d r_j + s) \times j}$, which amounts to
$(n + dj)(j+d) + sj \approx n (j+d) + sj$ numbers. Note that the approximation
$n + dj \approx n$ holds in any reasonable large-scale REP. In contrast,
RK stores $\mathbf{U}_j$, which amounts to $(nd+s)j = nd j + sj$ numbers. Since, $(j+d) < dj$ for most reasonable choices of $j$ and degrees $d$ appearing in practice, we see that R-CORK is much more memory-efficient than RK. These memory costs are shown in Table \ref{costs}.

With respect to the computational costs, observe that for both R-CORK and RK the cost is the sum of (i) the shift-and-invert step and (ii) the orthogonalization steps. Let us analyze first the shift-and-invert steps. If the shift-and-invert step in RK, i.e., step 2 in Algorithm \ref{rat.krylov}, is performed by applying an unstructured solver to the $(nd +s) \times (nd +s)$ linear system $(\cA - \theta_j \cB) \mathbf{\hat{u}} = \cB \mathbf{u}_j$, then the cost of RK is much larger than the cost of R-CORK, since R-CORK solves this system with Algorithm \ref{solver.system} (removing step 4) which is much more efficient because requires the solution of smaller linear systems (essentially, see Remark \ref{rem.costlinsolver}, one of size $n\times n$ and two of size $s\times s$, which are very often extremely small). However, one can consider to perform the shift-and-invert step in RK with Algorithm \ref{solver.system}, but this is still somewhat more expensive than R-CORK, because for RK it is needed to perform step 4 of Algorithm \ref{solver.system}, with an additional cost of $2n(d-2)$ flops in each iteration (see Remark \ref{rem.costlinsolver}). A final important remark on the shift-and-invert step is that R-CORK involves the overhead cost of constructing $\mathbf{u}_j$ in step 2 of Algorithm \ref{RCORK}, which in RK is not needed. However, note that, as explained in previous sections, this construction can be performed as in CORK via a single matrix-matrix product, which allows for optimal efficiency and cache utilisation on modern computers \cite[p. 577]{Kressner}. Moreover, we emphasize that a traditional construction of $\mathbf{u}_j$ in R-CORK costs $\cO(dn r_j) = \cO(dn (j+d))\approx \cO(dnj)$ flops at iteration $j$, which added to the orthogonalization cost of R-CORK discussed below would give a cost of the same order of the orthogonalization cost of RK.

Finally, we discuss the orthogonalization costs of RK and R-CORK. In RK, the orthogonalization is performed in steps 3-4-5 of Algorithm \ref{rat.krylov} and its cost is well-known to be $\cO(j(nd+s)) = \cO(jnd+js)$ flops at iteration $j$, which amounts to $\cO(j^2 n d+j^2 s)$ flops in the first $j$ iterations (see Table \ref{costs}). In R-CORK, the orthogonalization is performed in steps 5-6-7-8 of Algorithm \ref{RCORK}. At iteration $j$, the cost of step 5 is $\cO(r_j n) = \cO((j+d) n)$ flops, the cost of step 6 is $\cO(r_j d) = \cO((j+d) d)$ flops, which is negligible with respect to the cost of step 5, the cost of step 7 is $\cO(j (dr_j +s)) = \cO(j d (j+d) + j s)$ flops, and the cost of step 8 is $\cO(dr_j +s) = \cO(d (j+d) +s)$ flops. Therefore, the total cost at iteration $j$ of the orthogonalization in R-CORK is $\cO((n + j d) (j+d) + j s) \approx \cO(n  (j+d) + j s)$, where we have used again the approximation $n + j d \approx n$, which gives $\cO(j^2 n + j d n  + j^2 s)$ flops in the first $j$ iterations (see Table \ref{costs}). Observe that the orthogonalization cost of RK includes the large term $j^2 n d$ which is not present in the cost of R-CORK. Therefore, the orthogonalization cost of R-CORK is considerably smaller than the one of RK.

In Table \ref{costs}, we summarize the comparison of the costs between R-CORK and RK.

\begin{table}[H]
\begin{center}
\begin{tabular}{c|c|c}
& Classical rational Krylov method & R-CORK method\\ \hline
Orthogonalization cost & $\cO(j^2 n d + j^2 s)$ & $\cO(j^2 n + j d n  + j^2 s)$  \\ \hline
Memory cost & $ nd j+sj$ & $ n(j+d)+sj$
\end{tabular}
\end{center}
\caption{Orthogonalization and memory costs for classical rational Krylov method and R-CORK method after $j$ iterations.}\label{costs}
\end{table}

\section{Implicit restarting in R-CORK} \label{sect.implicit}
Practical implementations of any Krylov-type method for computing eigenvalues of large-scale problems require effective restarting strategies. The goal of this section is to develop an implicit restarting strategy for R-CORK that restarts both $Q_j$ and $\mathbf{R}_j$ in the compact representation of $\mathbf{U}_j$ in \eqref{blocks}-\eqref{defQ.R}. Since R-CORK shares many of the properties of CORK, the results of this section are similar to those in \cite[Section 6]{cork}, which in turn are based on implicit restarting procedures for classical rational Krylov methods \cite{implicit} and on the Krylov-Schur restart developed for TOAR in \cite[Section 4.2]{Kressner}.

Following the Krylov-Schur spirit \cite{stewart-paper-restart} (see also \cite[Section 5.2]{stewart-book}), the restarting technique we propose transforms  first the matrices $\underline{H}_j$ and $\underline{K}_j$ in \eqref{cork.relation} to (quasi)triangular form, in order to reorder the Ritz values and to preserve the desired ones with a rational Krylov subspace of smaller dimension. Second, by  representing the new smaller Krylov subspace in its compact form in an efficient way, the implicit restart of R-CORK is completed. The main difference of the process described below with respect to the implicit restarting in \cite[Section 6]{cork} is that here we need to add a new block of size $s\times s$ corresponding to the rational part of $R(\la)$ in \eqref{newreprat}.

Suppose that after $j$ iterations, we have the rational Krylov recurrence relation in its compact form as in \eqref{cork.relation}
\begin{equation}\label{full.rel}
\AA \mathbf{Q}_{j+1}\mathbf{R}_{j+1}\underline{H}_j = \BB \mathbf{Q}_{j+1}\mathbf{R}_{j+1}\underline{K}_j,
\end{equation}
and we want to reduce this representation to a smaller compact rational decomposition of size $p$, $p<j$, this is
\begin{equation*}
\AA \mathbf{Q}_{p+1}^+ \mathbf{R}_{p+1}^+ \underline{H}_p^+ = \BB \mathbf{Q}_{p+1}^+ \mathbf{R}_{p+1}^+ \underline{K}_p^+, \quad p<j.
\end{equation*}
For this purpose, we consider the generalized Schur decomposition:
\begin{eqnarray}
H_j&=&\left[ \begin{array}{cc} Y_p & Y_{j-p} \end{array}\right]\left[ \begin{array}{cc} T_{p\times p}^{(H)} & * \\ 0 & T^{(H)}_{(j-p)\times (j-p)}\end{array}\right] \left[\begin{array}{c} Z_p^* \\ Z_{j-p}^*\end{array}\right], \label{schur.H} \\
K_j&=&\left[ \begin{array}{cc} Y_p & Y_{j-p} \end{array}\right]\left[ \begin{array}{cc} T^{(K)}_{p\times p} & * \\ 0 & T^{(K)}_{(j-p)\times (j-p)}\end{array}\right] \left[\begin{array}{c} Z_p^* \\ Z_{j-p}^*\end{array}\right], \label{schur.K}
\end{eqnarray}
where $H_j$ and $K_j$ are the $j \times j$ upper Hessenberg matrices obtained by removing the last row of $\underline{H}_j$ and $\underline{K}_j$ respectively, $Y:=\left[ \begin{array}{cc} Y_p & Y_{j-p} \end{array}\right]$, $Z:=\left[ \begin{array}{cc} Z_p & Z_{j-p} \end{array}\right] \in\BC^{j\times j}$ are unitary matrices with $Y_p$, $Z_p \in \BC^{j\times p}$, $Y_{j-p}$, $Z_{j-p}\in\BC^{j\times (j-p)}$
and $T^{(H)}:=\left[ \begin{array}{cc} T_{p\times p}^{(H)} & * \\ 0 & T^{(H)}_{(j-p)\times(j-p)}\end{array}\right]$, $T^{(K)}:=\left[ \begin{array}{cc} T_{p\times p}^{(K)} & * \\ 0 & T^{(K)}_{(j-p)\times (j-p)}\end{array}\right] \in\BC^{j\times j}$ are upper (quasi)triangular matrices with $T_{p\times p}^{(H)}$, $T_{p\times p}^{(K)}\in\BC^{p\times p}$ and $T^{(H)}_{(j-p)\times (j-p)}$, $T^{(K)}_{(j-p)\times (j-p)}\in \BC^{(j-p)\times (j-p)}$. The $p < j$ Ritz values of interest are the eigenvalues of the pencil $T_{p\times p}^{(K)} - \la T_{p\times p}^{(H)}$.
By multiplying by $Z$ on the right the recurrence relation \eqref{full.rel} and using \eqref{schur.H} and \eqref{schur.K}, and considering the first $p$ columns, we obtain:
\begin{equation}\label{rel.reduced}
\AA \mathbf{Q}_{j+1}\mathbf{R}_{j+1}\left[\begin{array}{cc} Y_{p} & 0 \\ 0 & 1\end{array}\right]\left[\begin{array}{c} T_{p\times p}^{(H)} \\ h_{j+1,j}\tilde{z}^*\end{array}\right]= \BB \mathbf{Q}_{j+1}\mathbf{R}_{j+1}\left[\begin{array}{cc} Y_p & 0 \\ 0 & 1\end{array}\right]\left[\begin{array}{c} T_{p\times p}^{(K)} \\ k_{j+1,j}\tilde{z}^*\end{array}\right],
\end{equation}
where $\tilde{z}^*$ represents the first $p$ entries of the last row of $Z$.
By introducing the notation:
\begin{equation}\label{not.Y}
Y_1:=\left[\begin{array}{cc} Y_p & 0 \\ 0 & 1\end{array}\right]\in\BC^{(j+1)\times (p+1)} , \quad \underline{H}^+_p:=\left[\begin{array}{c} T_{p\times p}^{(H)} \\ h_{j+1,j}\tilde{z}^*\end{array}\right], \quad \underline{K}^+_p=\left[\begin{array}{c} T_{p\times p}^{(K)} \\ k_{j+1,j}\tilde{z}^*\end{array}\right] \in \BC^{(p+1)\times p},
\end{equation}
and defining $\mathbf{W}_{p+1}=\mathbf{R}_{j+1}Y_1$, we obtain
\begin{equation} \label{eq.auxxrestart}
\AA \mathbf{Q}_{j+1}\mathbf{W}_{p+1}\underline{H}^+_p = \BB \mathbf{Q}_{j+1}\mathbf{W}_{p+1}\underline{K}^+_p.
\end{equation}
Note that with this transformation, we reduce the size of the matrices $\underline{H}^+_p, \underline{K}^+_p$, and $\mathbf{W}_{p+1}$ with respect to $\underline{H}_j, \underline{K}_j$, and $\mathbf{R}_{j+1}$, and remove the Ritz values that are not of interest. However, observe that the large factor $\mathbf{Q}_{j+1}$ remains unchanged. In order to reduce the size of $\mathbf{Q}_{j+1}$, consider
\begin{equation*}
\mathbf{W}_{p+1} = \left[\begin{array}{c} W_{p+1}^{(1)} \\ \vdots \\ W_{p+1}^{(d)} \\ V_{j+1}Y_1 \end{array} \right], \quad W_{p+1}^{(i)} \in \BC^{r_{j+1}\times(p+1)}, \quad i=1,\dots,d,
\end{equation*}
and let $\omega$ be the rank of $[W_{p+1}^{(1)} \quad \cdots \quad W_{p+1}^{(d)}]$. The key observation is that although the matrices $\underline{H}^+_p, \underline{K}^+_p$ are no longer in Hessenberg form, the subspace spanned by the columns of $(\mathbf{Q}_{j+1}\mathbf{W}_{p+1})$ is still a rational Krylov subspace corresponding to $\cA - \la \cB$ \cite{implicit}. Therefore, we can apply Theorem \ref{th.rjRCORK} to $\mbox{span} \left\{ Q_{j+1} W_{p+1}^{(1)}, \ldots , Q_{j+1} W_{p+1}^{(d)} \right\} = Q_{j+1} \, \mbox{span} \left\{ W_{p+1}^{(1)}, \ldots , W_{p+1}^{(d)} \right\}$ to obtain that $\omega \leq d + p < d +j$. Then, we compute the economy singular value decomposition of:
\begin{eqnarray*}
[W_{p+1}^{(1)} \quad \cdots \quad W_{p+1}^{(d)}]&=&\cU\cS[\cV^{(1)} \quad \cdots \quad \cV^{(d)}],
\end{eqnarray*}
where $\cU \in \BC^{r_{j+1}\times \omega}$, $\cS \in \BC^{\omega\times \omega}$ and $\cV^{(i)}\in\BC^{\omega \times (p+1)}$ for $i=1,\dots,d$. Thus, by defining
\begin{equation*}
Q^+_{p+1}=Q_{j+1}\cU,  \quad \mathbf{R}_{p+1}^+=\left[ \begin{array}{c} \cS\cV^{(1)} \\ \vdots \\ \cS\cV^{(d)} \\ V_{j+1}Y_1 \end{array} \right], \quad \mathbf{Q}_{p+1}^+=\left[ \begin{array}{cccc} Q^+_{p+1} & & & \\ & \ddots & & \\ & & Q^+_{p+1} & \\ & & & I_s \end{array}\right],
\end{equation*}
we get from \eqref{eq.auxxrestart} the compact rational Krylov recurrence relation
\begin{equation} \label{eq.lastrestart}
\AA \mathbf{Q}_{p+1}^+ \mathbf{R}_{p+1}^+ \underline{H}^+_p= \BB \mathbf{Q}_{p+1}^+ \mathbf{R}_{p+1}^+ \underline{K}^+_p,
\end{equation}
with $p<j$. It is important to emphasize again that the matrices $\underline{H}^+_p$ and $ \underline{K}^+_p$ are no longer upper Hessenberg matrices, however, they contain the required Ritz values and the columns of $\mathbf{Q}_{p+1}^+ \mathbf{R}_{p+1}^+$ span a corresponding rational Krylov subspace. We continue the process by expanding \eqref{eq.lastrestart} with Algorithm \ref{RCORK} until we get a rational Krylov subspace of dimension $j$. The matrices $\underline{H}^+_j$ and $ \underline{K}^+_j$ obtained in this expansion are not in Hessenberg form, although their columns $p+1, \ldots , j$ have a Hessenberg structure (see \cite[p. 329]{stewart-book}). Then, the restarting process described in this section is applied again to get a new compact relation \eqref{eq.lastrestart} of ``size $p$''. This expansion-restarting procedure is cyclicly repeated until the prescribed stopping criterion is satisfied for a certain desired number, less than or equal to $p$, Ritz pairs.

\section{Numerical experiments} \label{sec.numexper}

In this section, we present two large-scale and sparse numerical examples to illustrate the efficiency of the R-CORK method. All reported experiments were performed using Matlab R2013a on a PC with a 2,2 GHz Intel (R) Core (TM) i7 processor, with 16 GB of RAM and DDR3 memory type, and with operating system macOS Sierra, version 10.12.1.

By following \cite[Section 8]{cork}, in the numerical experiments we plot the residuals at each iteration, with and without restarts, obtained by using the R-CORK method, the dimension of the subspace at each iteration for R-CORK, and the comparison of the memory storages of R-CORK and of the classical rational Krylov method applied directly to the linearization \eqref{linearrat}. We
also report on the number of iterations until convergence.

Inspired by the applications in \cite[Section 4]{SuBai}, we construct numerical experiments with prescribed eigenvalues and poles of a rational matrix $R(\lambda)$ represented as in \eqref{newreprat}. In order to measure the convergence of an approximate eigenpair $(\lambda, x)$ of $R(\la)$, we consider the relative norm residual:
\begin{equation}\label{res.ex}
E(\lambda,x)=\dfrac{\| R(\lambda)x\|_2}{(\sum_{i=0}^d{| \lambda |^{i}\| P_i\|_F}+\|E(C-\lambda D)^{-1}F^T\|_F)\| x\|_2}.
\end{equation}
Observe that the computation of $E(\lambda,x)$ involves matrices and vectors of size $n$ and, so, is expensive. Therefore, in actual practice, we recommend to test first the convergence through a cheap estimation of the residual of the linearized problem, i.e., $\|(\cA - \la \cB) \mathbf{z}\|_2$, involving only the small projected problem \eqref{solveHK}, and once such residual is sufficiently small to compute the residual \eqref{res.ex} every $q>1$ iterations instead of at each iteration. However, in our examples, we performed the computation of $E(\lambda,x)$ at each iteration for the purpose of illustration.

The computation of \eqref{res.ex} deserves some comments. Note first that it requires to recover the approximated eigenvector $x$ of $R(\la)$ from the approximated eigenvector $\mathbf{z}$ of the linearization $\cA - \la \cB$ in \eqref{linearrat} computed in step 10 of Algorithm \ref{RCORK}. This recovery, according to the first equation in \eqref{eigenv}, can be done by taking any of the first $d$ blocks of $\mathbf{z}$ if $\la \ne 0$. Since in our numerical examples the moduli of the approximate eigenvalues are larger than $1$, we have chosen the first block of $\mathbf{z}$ as approximate $x$. However, we recommend to choose the $d$th block if the moduli of the approximate eigenvalues are smaller than $1$. The calculation of the quantities $\|P_i\|_F$, $i=0,\dots,d$ needs to be performed only once and it is inexpensive since the matrices $P_i$ are sparse in practice. Finally, to compute the expression $\|E(C-\lambda D)^{-1}F^T\|_F$ on the denominator in \eqref{res.ex}, we use
\begin{eqnarray*}
\|E(C-\lambda D)^{-1}F^T\|_F^2&=&\mbox{trace}((E(C-\lambda D)^{-1}F^T)^*E(C-\lambda D)^{-1}F^T), \\ &=&\mbox{trace}((E^*E)(C-\lambda D)^{-1}(F^T \bar F)(C-\lambda D)^{-*}),
\end{eqnarray*}
which only involves the matrices $E^*E$, $F^T \bar F$, $(C-\lambda D)^{-1}$ and $(C-\lambda D)^{-*}$ of size $s\times s$. Since in many application $s\ll n$, this computation is usually inexpensive.

\begin{numexamp}\label{numexamp1}
We construct a REP of the type arising from the free vibrations of a structure if one uses a viscoelastic constitutive relation to describe the behavior of a material \cite{visco,SuBai}. The REPs of this type have the following structure:
\begin{equation}\label{exp1}
R(\lambda)x=\left(\lambda^2M+K-\sum_{i=1}^k{\dfrac{1}{1+b_i\lambda}\Delta G_i}\right)x=0,
\end{equation}
where the mass and stiffness matrices $M$ and $K$ are real symmetric and positive definite, $b_j$ are relaxation parameters over $k$ regions, and $\Delta G_j$ is an assemblage of element stiffness matrices over the region with the distinct relaxation parameters.  As in \cite{SuBai}, we consider the case where $\Delta G_i=E_iE_i^T$ and $E_i\in\BR^{n\times s_i}$. By defining
\begin{equation*}
E=[E_1, E_2, \dots, E_k], \quad D=\mbox{diag}(b_1I_{s_1},b_2I_{s_2},\dots,b_kI_{s_k}),
\end{equation*}
the REP \eqref{exp1} can be written in the form \eqref{newreprat}:
\begin{equation*}
(\lambda^2M+K-E(I+\lambda D)^{-1}E^T)x=0.
\end{equation*}
In our particular example, we consider the case with one region and one relaxation parameter $b_1=-1$. The construction of the matrices $M$ and $K$ in our example proceeds as follows: construct first $R_1(\lambda)=\lambda^2A_2+A_0-e_{10000}(1-\lambda)^{-1}(e_{10000})^T$, with $A_2,A_0\in\mathbb{R}^{10000\times 10000}$ diagonal and positive definite matrices and $e_{10000}$ the last column of $I_{10000}$. This structure allows to prescribe easily the eigenvalues for $R_1 (\lambda)$. Then, we consider the following invertible tridiagonal matrix $P$
\begin{equation*}
P=\left[ \begin{array}{cccc} 1 & \frac{1}{2} & & \\ \frac{1}{3} & 1 & \ddots & \\ & \ddots & \ddots & \frac{1}{2} \\ & & \frac{1}{3} & 1 \end{array}\right],
\end{equation*}
and finally construct $R(\lambda)=PR_1(\lambda)P^T$. Since $P$ is invertible, the eigenvalues of $R(\lambda)$ and $R_1(\lambda)$ are the same. By using this procedure, we have constructed the REP
\begin{equation}\label{re.exp1}
R(\lambda)\, x=\left(\lambda^2M+K-p_{10000}(1-\lambda)^{-1}(p_{10000})^T\right) \, x=0,
\end{equation}
where $M:=PA_2P^T$, $K:=PA_0P^T\in\mathbb{R}^{10000\times 10000}$ are symmetric, positive definite, and pentadiagonal matrices, and $p_{10000}\in\mathbb{R}^{10000}$ represents the last column of the matrix $P$.

In this example we are interested in computing the $20$ eigenvalues of \eqref{re.exp1} with negative imaginary part and with largest absolute value of the negative imaginary part.  To aim our goal, we use $3$ cyclically repeated shifts in the rational Krylov steps and a random unit real vector as an initial vector. The reader can see the approximate eigenvalues computed by R-CORK and the chosen shifts in Figure \ref{numexamp1}(a). We first solve the REP \eqref{re.exp1} by using Algorithm \ref{RCORK} without restart, and after $85$ iterations, we find the required eigenvalues with a tolerance \eqref{res.ex} of $10^{-10}$. The convergence history is shown in Figure \ref{numexamp1}(b). In Figure \ref{numexamp1}(d), we plot  $r_j$, the rank of $Q_j$ at the iteration $j$, and $j$, the dimension of the Krylov subspace. Since we did not perform restart, we can see that both, $r_j$ and $j$ increases with the iteration count $j$ and that $r_j = j +1$, as expected since the degree of the polynomial part of \eqref{re.exp1} is $d=2$. Figure \ref{numexamp1}(f) displays the comparison between the cost of memory storage of both the R-CORK method, by using Algorithm \ref{RCORK}, and the classical rational Krylov method, by using Algorithm \ref{rat.krylov}. From this figure, we can see that the R-CORK method requires approximately half of the memory storage that the classical rational Krylov method, which is consistent with the degree $2$ of the polynomial part of \eqref{re.exp1}.

Next, we apply  Algorithm \ref{RCORK} to the REP \eqref{re.exp1} combined with the implicit restarting introduced in Section \ref{sect.implicit}. We choose the maximum dimension of the subspace $m=45$, which is reduced after each restart to dimension $p=30$ to compute the $20$ required eigenvalues. The convergence history of the eigenpairs $(\lambda,x)$ computed by this restarted R-CORK method is shown in Figure \ref{numexamp1}(c). After $3$ restarts and $81$ iterations, the $20$ required eigenvalues have been found with a tolerance \eqref{res.ex} of $10^{-10}$. In Figure \ref{numexamp1}(e) the reader can see the rank of $Q_j$ at the $j$-th iteration and it can be seen that with restart, the relation between $j$ and $r_j$ continues the same. Finally, in Figure \ref{numexamp1}(g) we plot the memory storage for R-CORK and classical rational Krylov, and it can be observed that the memory cost for the R-CORK method is a factor close to 2 smaller than the memory cost obtained by the classical rational Krylov method.
\end{numexamp}

\begin{figure}
\begin{center}
\subfigure[Eigenvalues.]{\includegraphics[scale=0.23]{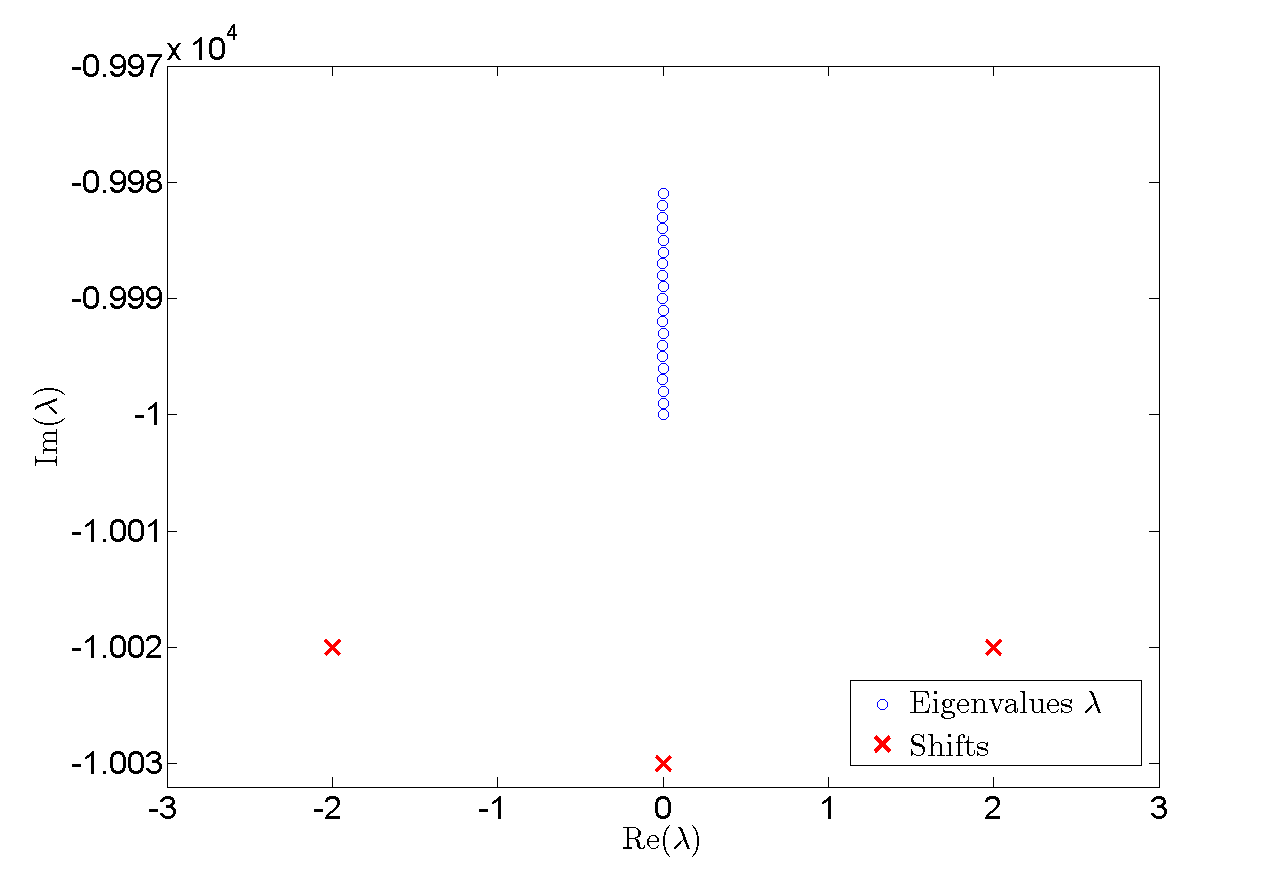}}
\end{center}\vspace{-0.5 cm}
\centering
\subfigure[Convergence history without restart.]{\includegraphics[scale=0.23]{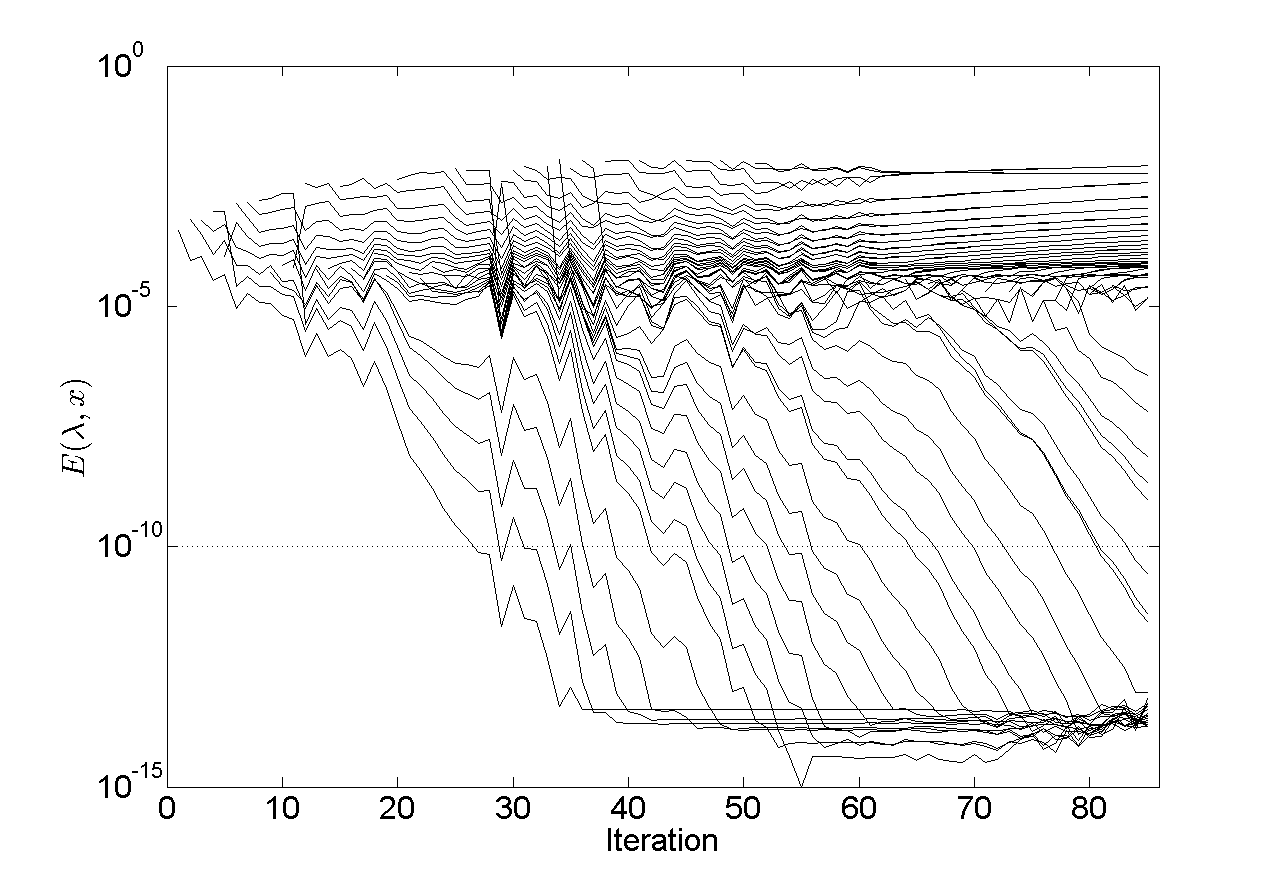}}\vspace{-0.1cm}
\subfigure[Convergence history with restart.]{\includegraphics[scale=0.23]{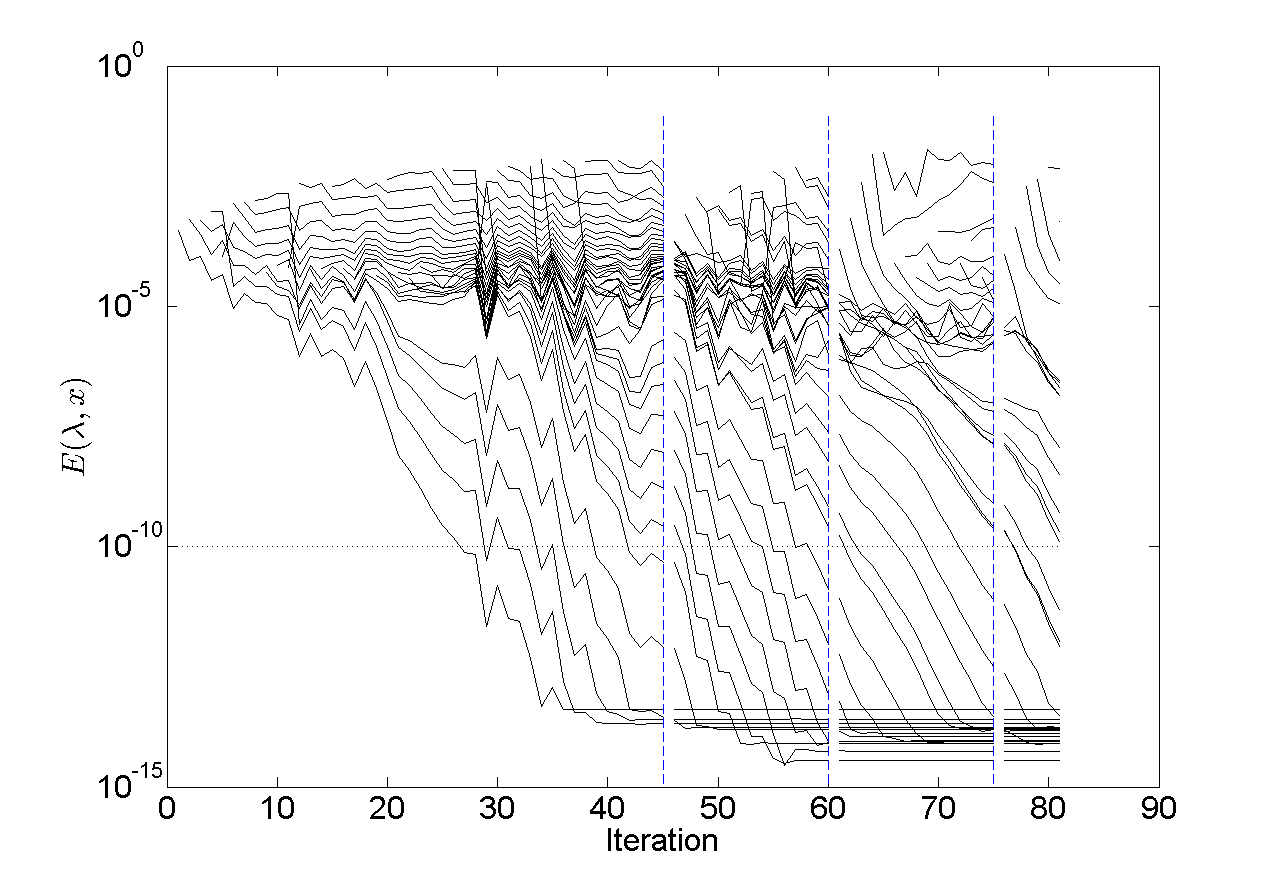}}\vspace{-0.1cm}
\subfigure[Dimension of the subspace without restart.]{{\includegraphics[scale=0.23]{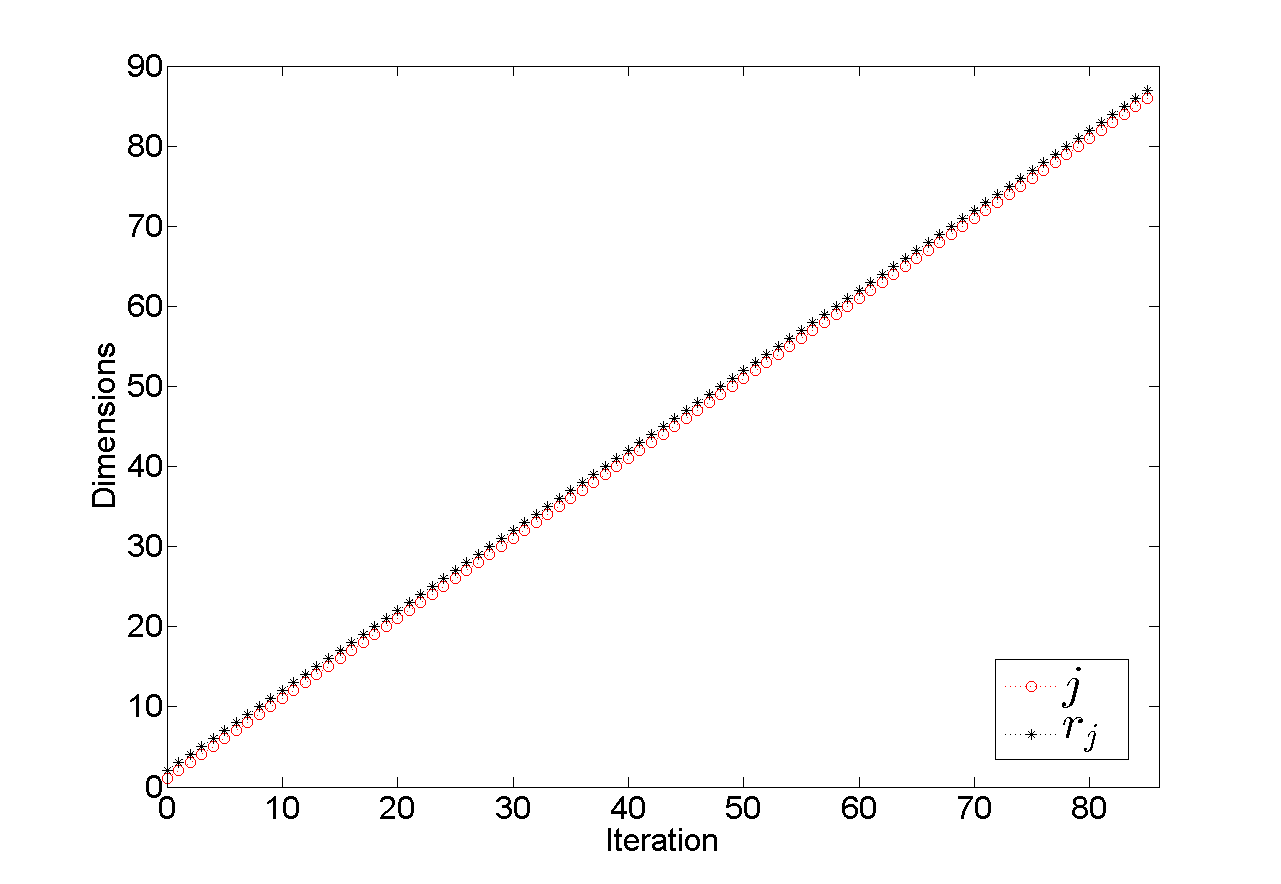}}}\vspace{-0.1cm}
\subfigure[Dimension of the subspace with restart.]{{\includegraphics[scale=0.23]{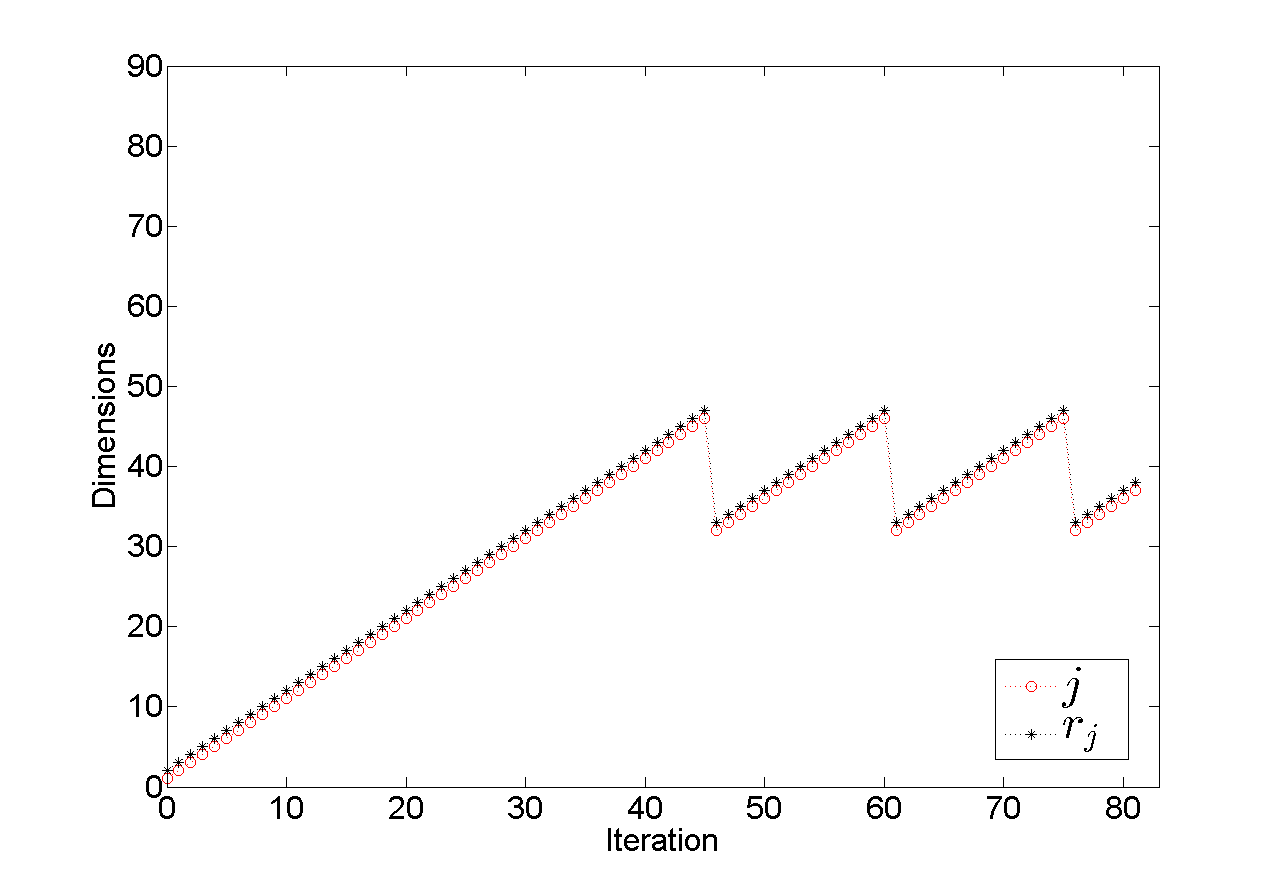}}}\vspace{-0.1cm}
\subfigure[Memory cost without restart.]{{\includegraphics[scale=0.23]{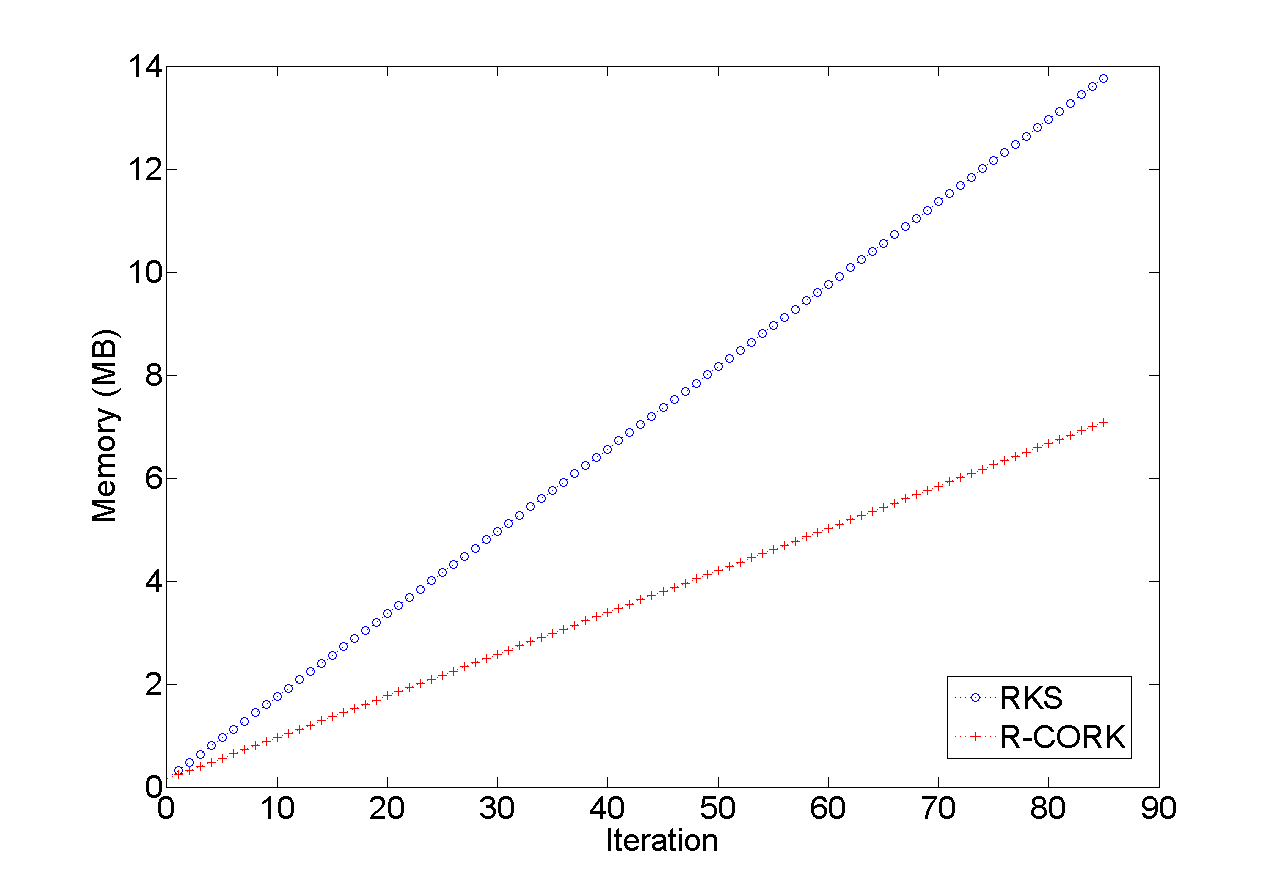}}}\vspace{-0.1cm}
\subfigure[Memory cost with restart.]{{\includegraphics[scale=0.23]{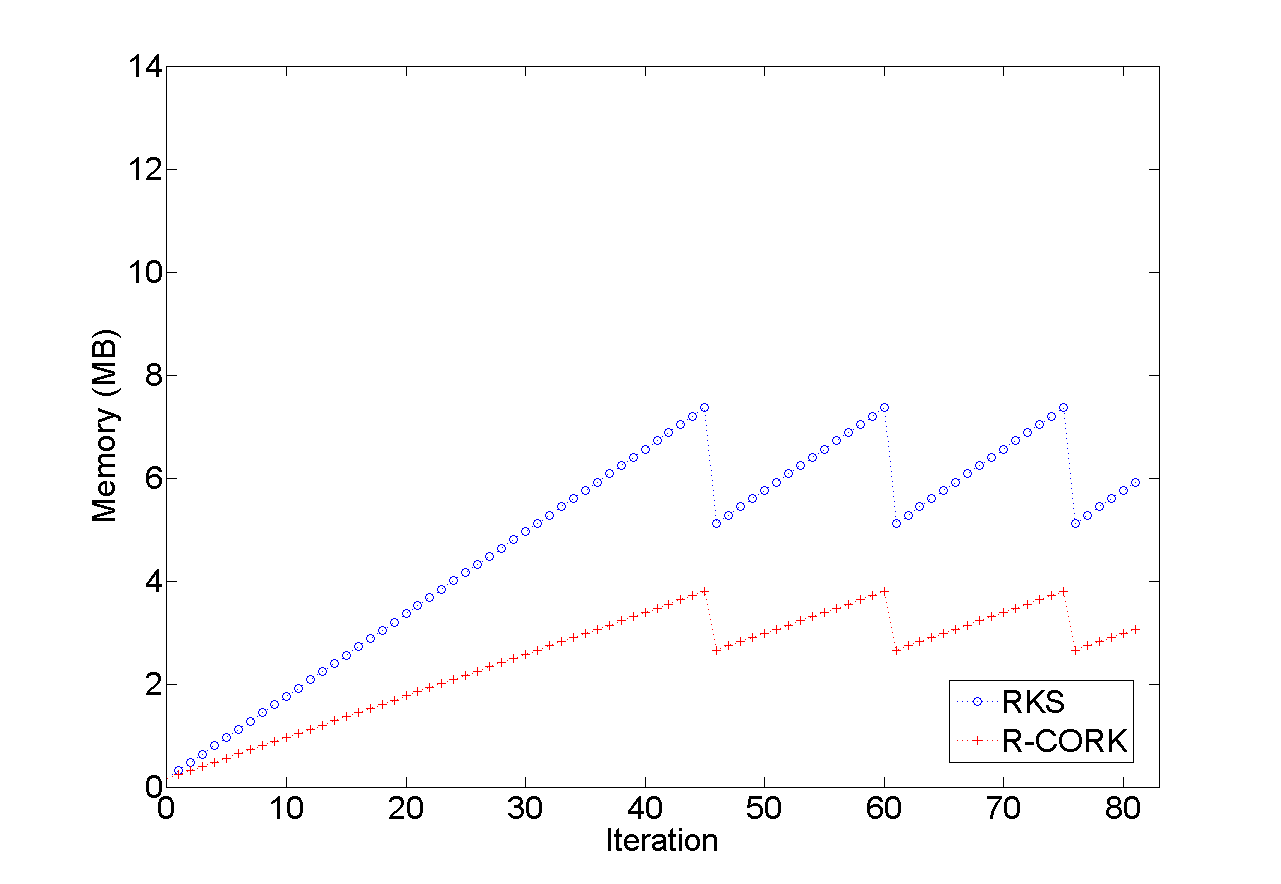}}}\\
\caption{Numerical experiment 5.1.}
\end{figure}

\begin{numexamp}\label{numexamp2}
For this numerical example, we consider an academic REP of size $5000\times 5000$ and with the degree of its polynomial part equal to $3$, i.e., a REP of the form
\begin{equation}\label{numex2}
R(\lambda)=\lambda^3A_3+\lambda^2A_2+\lambda A_1+A_0-E(C-\lambda D)^{-1}F^T{\color{magenta}.}
\end{equation}
The coefficient matrices of $R(\lambda)$ in \eqref{numex2}  were constructed in a similar way that in the numerical experiment \ref{numexamp1}: first, we consider a rational matrix $R_2(\lambda)=\lambda^3P_3+\lambda^2P_2+\lambda P_1+P_0-E_0(C-\lambda D)^{-1}F_0^T$ with prescribed eigenvalues, where $P_i\in \BR^{5000\times 5000}$ are diagonal matrices, $E_0=[e_1+e_2,\quad e_5+e_6], \quad F_0=[ e_{4997}+e_{4998} ,\quad e_{4999}+e_{5000}]\in\BR^{5000\times 2}$, with $e_i$ the $i$th canonical vector of size $5000$,  and
\begin{equation*}
C=\left[\begin{array}{cc} 105 & 0 \\ 0 & -105 \end{array}\right], \quad
D=I_2,
\end{equation*}
and then we define $R(\lambda)=PR_2(\lambda)Q$, where
\begin{equation*}
P=\left[ \begin{array}{ccccc} 1 & \frac{1}{2} & \frac{1}{3} & & \\ -\frac{1}{4} & \ddots& \ddots & \ddots& \\ -\frac{1}{5} & \ddots & \ddots & \ddots &\frac{1}{3} \\ & \ddots & \ddots & \ddots & \frac{1}{2} \\ & & -\frac{1}{5} & -\frac{1}{4} & 1 \end{array}\right], \quad
Q=\left[ \begin{array}{cccc} -1 & -\frac{1}{3} & & \\ \frac{1}{2} & \ddots &  \ddots & \\ & \ddots & \ddots & -\frac{1}{3} \\&  & \frac{1}{2} & -1 \end{array} \right] \in\BR^{5000\times 5000}.
\end{equation*}
The goal of this example is to compute the 30 eigenvalues closest to zero. In this situation, it is natural to choose zero as a fixed shift. In Figure \ref{numexamp2}(a), the approximate eigenvalues computed by R-CORK are displayed. By starting with a random unit complex vector, first we apply R-CORK without restarting, and after 83 iterations, the desired eigenvalues are obtained with a tolerance \eqref{res.ex} of $10^{-12}$. The convergence history can be seen in Figure \ref{numexamp2}(b). In Figure \ref{numexamp2}(d), we see that the relation $r_j<j+d$ with $j$ the number of iterations also holds in this example, though in this case with $d=3$ since this is the degree of the polynomial part in \eqref{numex2}. Figure \ref{numexamp2}(f) shows the memory costs of R-CORK and classical rational Krylov. It is observed that the reduction in cost of R-CORK is approximately a factor of $3$, i.e., the degree of the polynomial part of \eqref{numex2}.

As a final example, we solve \eqref{numex2} by using R-CORK combined with  restarting and taking a maximum subspace dimension $m=60$ which is reduced to $p=40$ after every restart. The convergence history is shown in Figure \ref{numexamp2}(c), where it is observed that after 91 iterations and 2 restarts, the 30 eigenvalues closest to zero have been found with a tolerance \eqref{res.ex} of $10^{-12}$. Despite the fact that a few more iterations are needed with restart than without restart, we see in Figure \ref{numexamp2}(e) that we are using a Krylov subspace of much smaller dimension than without restart to compute the eigenpairs. In addition, we emphasize that Figure \ref{numexamp2}(e) shows that for this particular example $r_j < j$ after the restarts, which illustrates that in practice the upper bound $d+j-1$ in Theorem \ref{th.rjRCORK} is not always attained and that the memory efficiency of R-CORK can be larger than the one theoretically expected.  Finally, the comparison of the memory costs for the R-CORK and for the classical rational Krylov methods is plotted in Figure \ref{numexamp2}(g), where we see again that the cost of R-CORK is approximately a factor $d=3$ smaller.
\begin{figure}
\begin{center}
\subfigure[Eigenvalues.]{\includegraphics[scale=0.23]{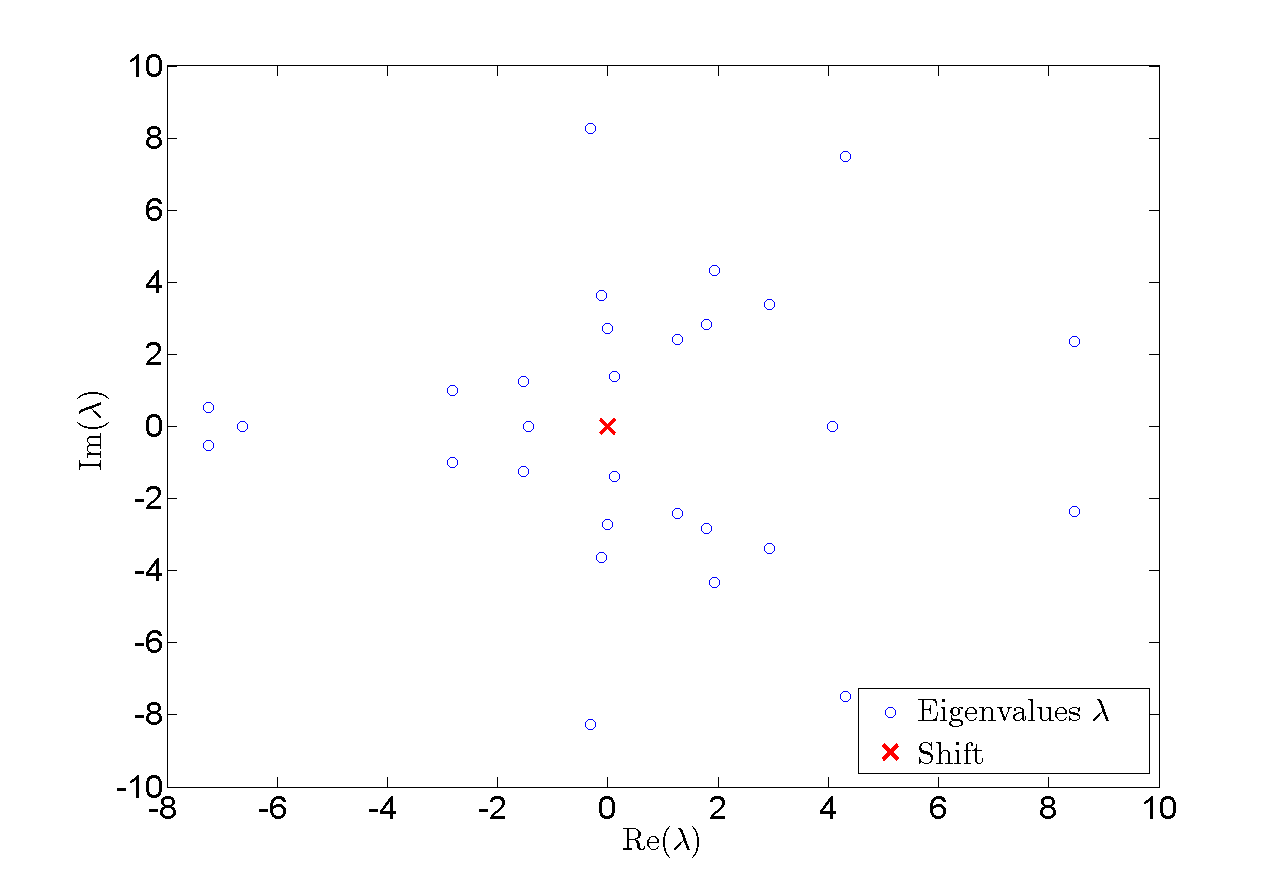}}
\end{center}\vspace{-0.5 cm}
\centering
\subfigure[Convergence history without restart.]{\includegraphics[scale=0.23]{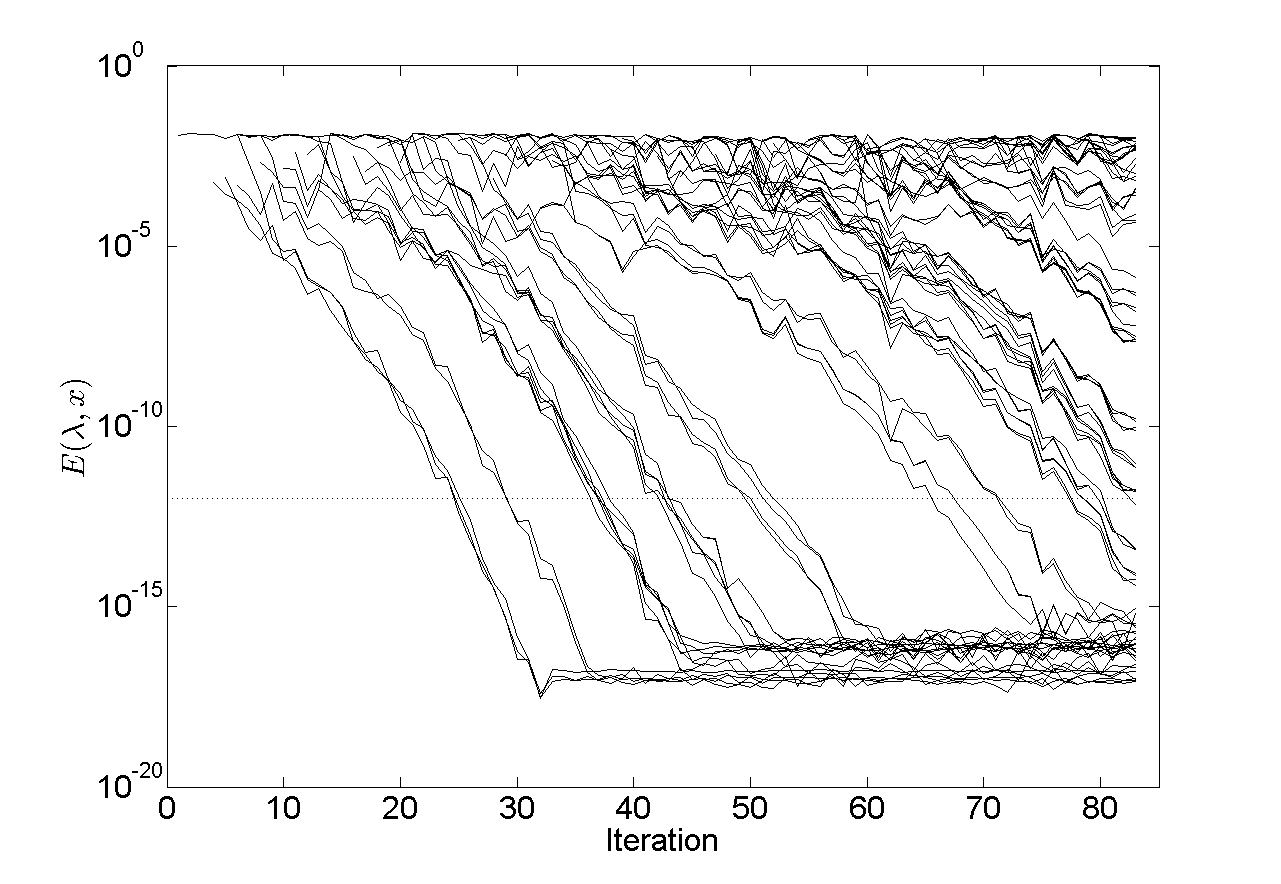}}\vspace{-0.1cm}
\subfigure[Convergence history with restart.]{\includegraphics[scale=0.23]{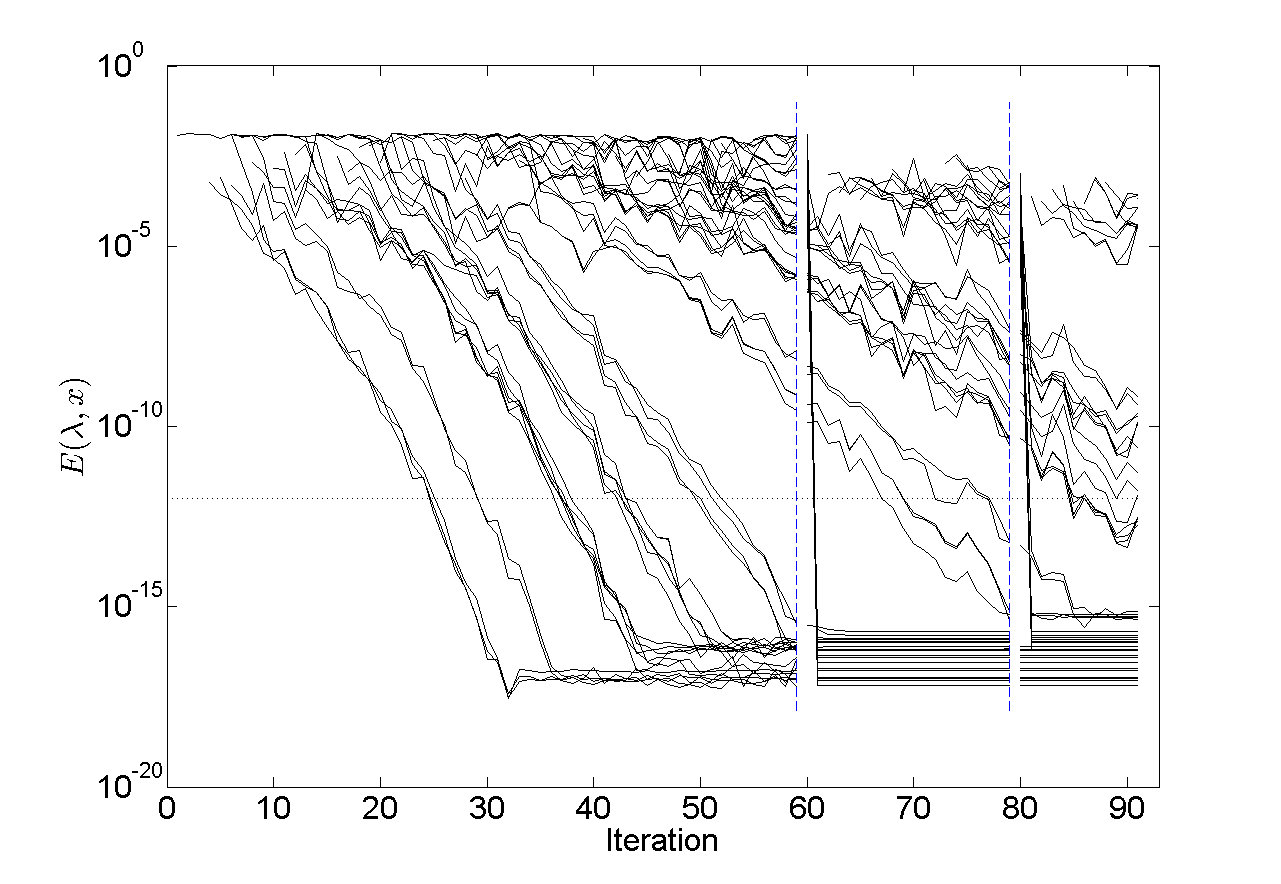}}\vspace{-0.1cm}
\subfigure[Dimension of the subspace without restart.]{{\includegraphics[scale=0.23]{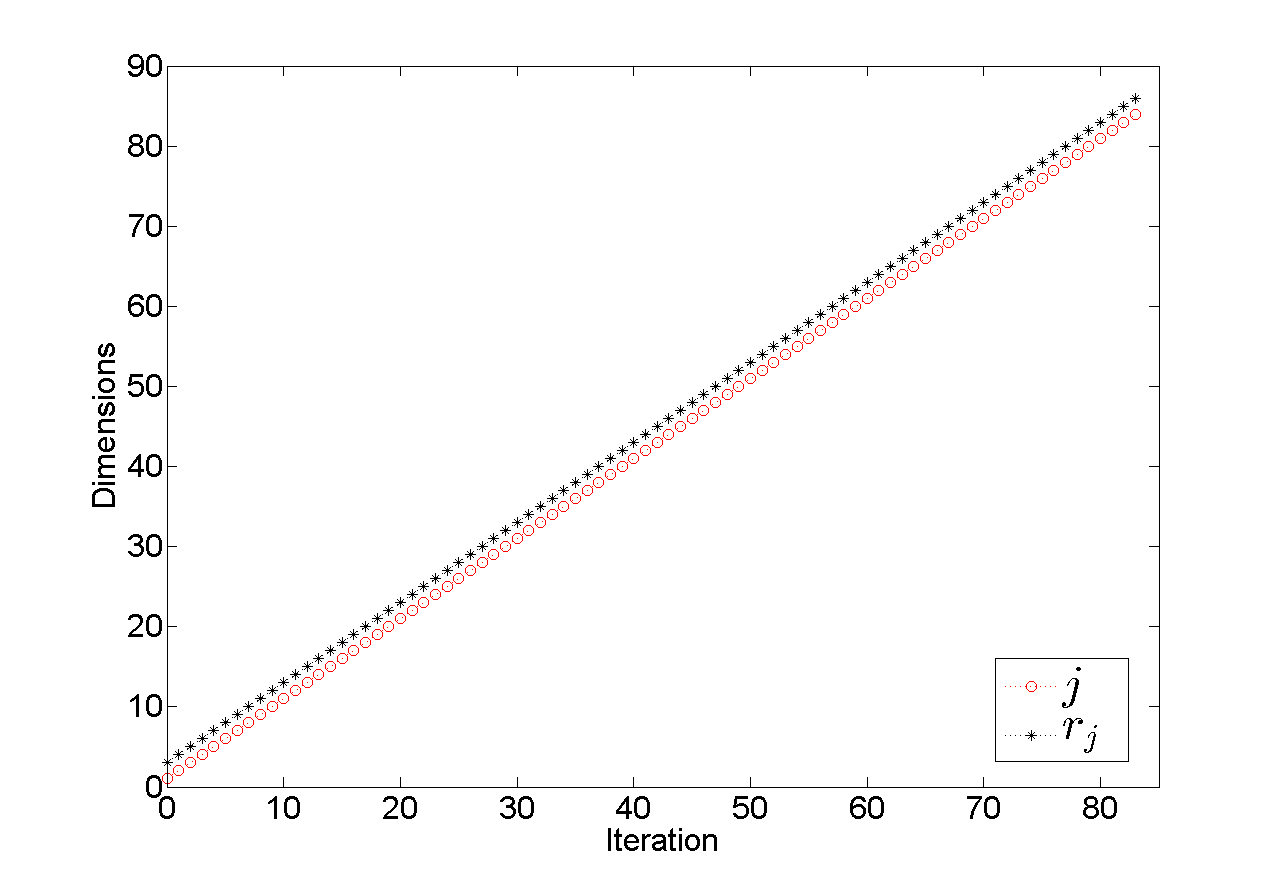}}}\vspace{-0.1cm}
\subfigure[Dimension of the subspace with restart.]{{\includegraphics[scale=0.23]{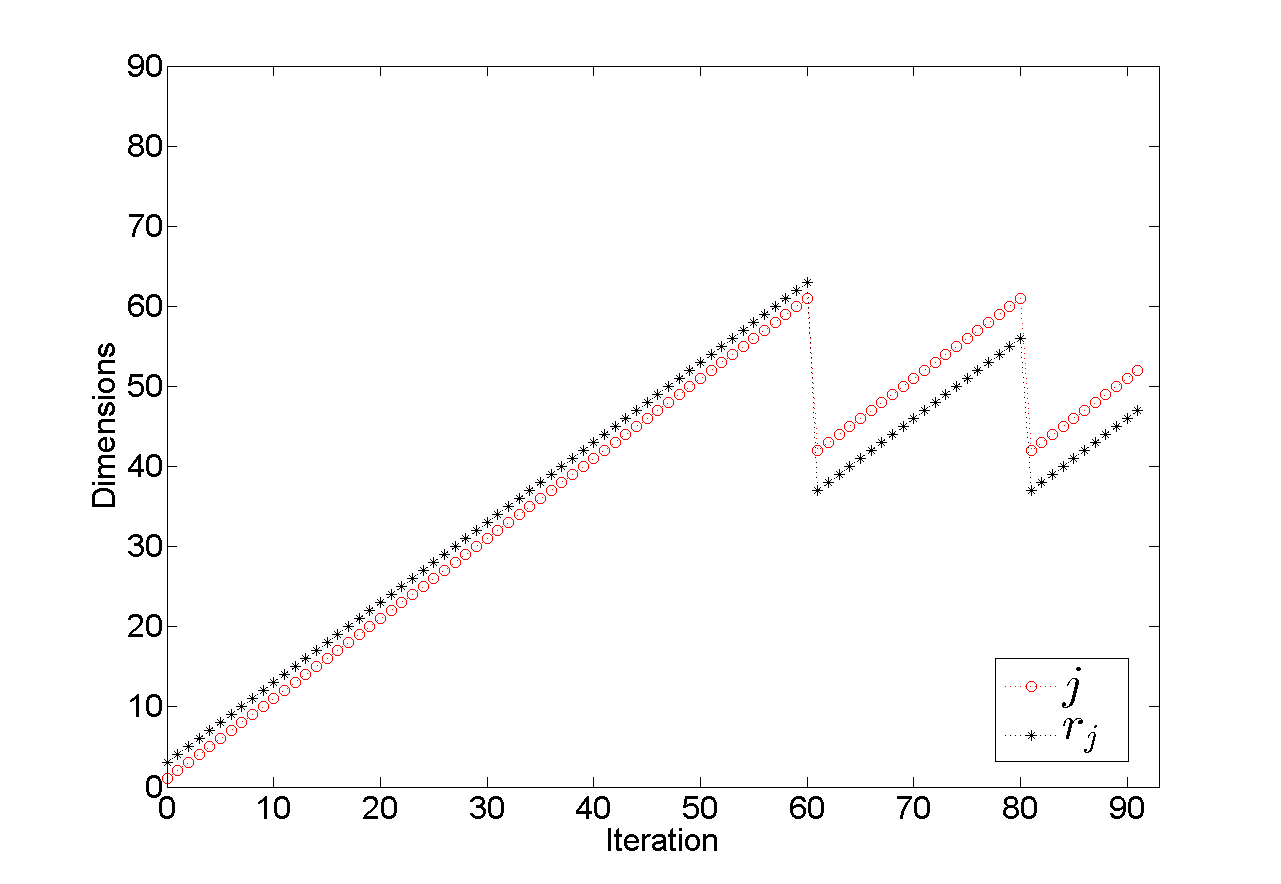}}}\vspace{-0.1cm}
\subfigure[Memory cost without restart.]{{\includegraphics[scale=0.23]{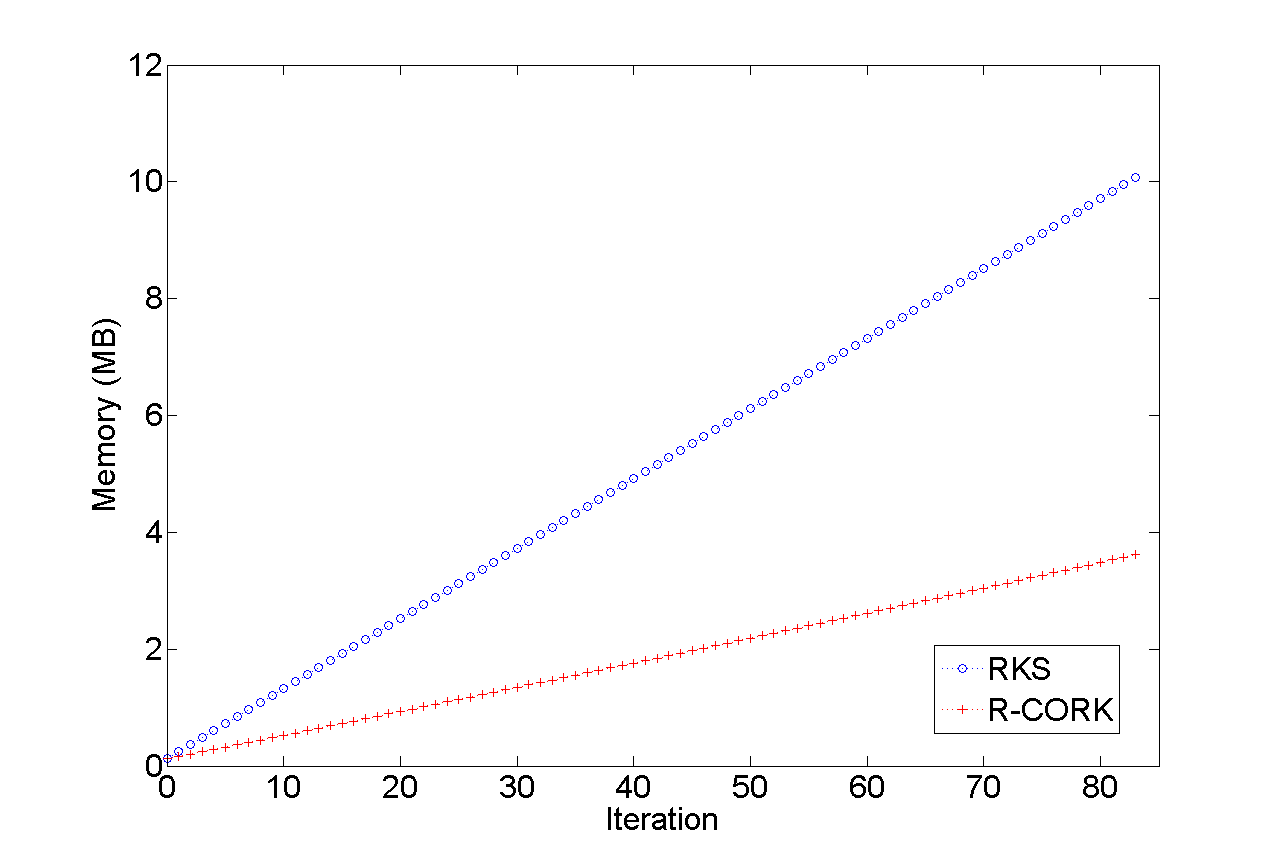}}}\vspace{-0.1cm}
\subfigure[Memory cost with restart.]{{\includegraphics[scale=0.23]{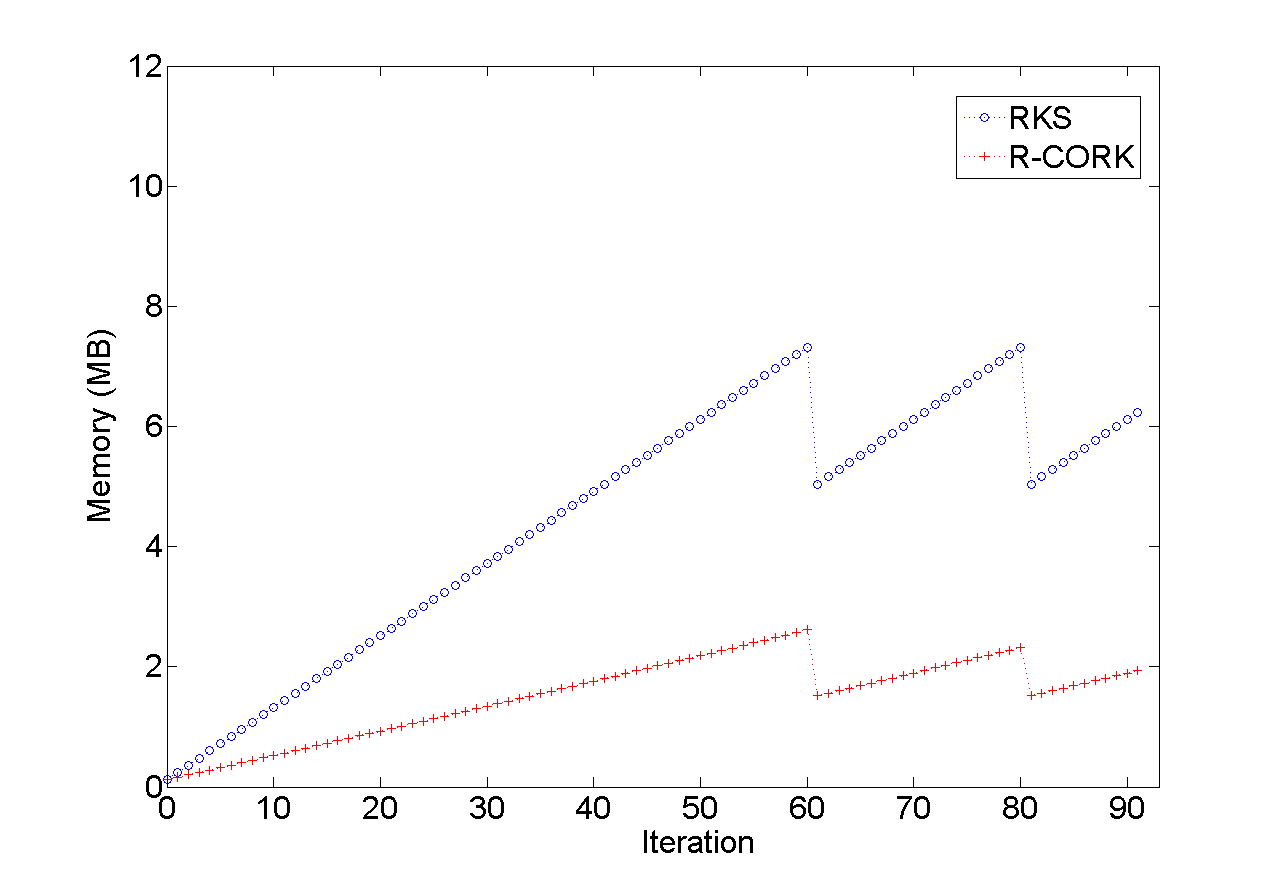}}}\\
\caption{Numerical experiment 5.2}\label{fig1}
\end{figure}
\end{numexamp}


\section{Conclusions and lines of future research} \label{sec.conclu}
In this paper, we have introduced the R-CORK method for solving large-scale rational eigenvalue problems that are represented as the sum of their polynomial and strictly proper parts as in \eqref{newreprat}. The first key idea is that R-CORK solves the generalized eigenvalue problem associated to the Frobenius companion-like linearization \eqref{linearrat} previously introduced in \cite{SuBai}. The second key idea is that R-CORK is a structured version of the classical rational Krylov method for solving generalized eigenvalue problems that takes advantage of the particular structure of \eqref{linearrat}. This structure allows us to represent the orthonormal bases of the rational Krylov subspaces of \eqref{linearrat} in a compact form involving less parameters than the bases of rational Krylov subspaces of the same dimension corresponding to unstructured generalized eigenvalue problems of the same size as the considered linearization. In addition, this compact form can be efficiently and stably updated in each rational Krylov iteration by the use of two levels of orthogonalization in the spirit of the TOAR \cite{charlaTOAR,Kressner} and the CORK \cite{cork} methods for large-scale polynomial eigenvalue problems.

The combined use of the compact representation of rational Krylov subspaces and the two levels of orthogonalization in R-CORK reduces significantly the orthogonalization and the memory costs with respect to a direct application of the classical rational Krylov method to the linearization \eqref{linearrat}. If $n\times n$ is the size of the rational eigenvalue problem, $j$ is the maximum dimension of the considered Krylov subspaces of the linearization, $d$ is the degree of the matrix polynomial $P(\la)$ in \eqref{newreprat}, and $s\times s$ is the size of the pencil $(C-\la D)$ appearing in \eqref{newreprat} (note that if $s \leq n$, then $s$ is essentially the rank of the strictly proper part of the rational matrix $R(\la)$), then the reduction in costs of R-CORK is appreciable whenever $jd \ll n$ and very considerable if, moreover, $s \ll n$ and $d < j$. In this situation, after $j$ iterations, the orthogonalization cost of R-CORK is $\cO (j^2 n)$, while the one of classical rational Krylov is $\cO (j^2 n d)$, and the memory cost of R-CORK is approximately $n j$ numbers, while the one of classical rational Krylov is $n d j$. These reductions can be combined with an structured implementation of a Krylov-Schur implicit restarting adapted to the compact representation used by R-CORK, which allows us to keep the dimension of the Krylov subspaces moderate without essentially increasing the number of iterations until convergence. The performed numerical experiments confirm all these good properties of R-CORK.

Since many linearizations of rational matrices different from \eqref{linearrat} have been developed very recently \cite{behera,ADMZ2016} and some of them include the option of considering that the matrix polynomial $P(\la)$ in \eqref{newreprat} is expressed in non-monomial bases, an interesting line of future research on the numerical solution of rational eigenvalue problems is to investigate the potential extension of the R-CORK strategy to other linearizations.

\bigskip
\noindent
{\bf Acknowledgements.} The authors sincerely thank Roel Van Beeumen and Karl Meerbergen for answering patiently and very carefully many questions on the CORK method they developed in \cite{cork}. Their help has been very important for improving somes parts of this manuscript.

\nocite{*}
\bibliographystyle{plain}

\end{document}